\documentclass[11pt]{amsart} 
\usepackage{graphicx, epsfig}
\usepackage{bbding} 
\usepackage{amssymb,amsfonts,amsmath, amscd, mathrsfs,amsthm,wasysym,pifont,stmaryrd,inputenc,ifthen, pifont}
\usepackage{epstopdf,yfonts, setspace, mdwlist, calc, float, comment}
\usepackage{wrapfig,color,microtype, enumitem}
\usepackage{mathrsfs}
\usepackage{cutwin}
\usepackage{lipsum}
\usepackage{tikz,tikz-cd}
\usetikzlibrary{decorations.pathmorphing}

\usepackage{xcolor}
\usepackage{contour}

\usepackage{hyperref}
\usepackage{nameref,zref-xr}
\zxrsetup{toltxlabel}
\zexternaldocument*{Part_II_Semi_simple_theory}

\title{Acylindricity in Higher Rank\\ \vspace{5pt} Part I: Fundamentals}

\newcounter{scomments}

\newcounter{tcomments}

\makeatletter
\let\@wraptoccontribs\wraptoccontribs
\makeatother

\newcounter{cl}
\newcounter{clno}
\newtheorem{theorem}{Theorem}[section]
\newtheorem*{question}{Question}

\newtheorem{prop}[theorem]{Proposition}

\newtheorem{cor}[theorem]{Corollary}
\newtheorem{lemma}[theorem]{Lemma}

\theoremstyle{definition}
\newtheorem{defn}[theorem]{Definition}

\newtheorem{remark}[theorem]{Remark}

\newtheorem{thmA}{Theorem}

\DeclareSymbolFont{bbsymbol}{U}{bbold}{m}{n}
\DeclareMathSymbol{\bbcomma}{\mathbin}{bbsymbol}{"2C}

\newcommand{\Z}{\mathbb Z}
\newcommand{\consttwo}{E}
\newcommand{\op}{\operatorname}

\renewcommand{\~}[1]{\overline{#1}}
\renewcommand{\geq}{\geqslant}
\renewcommand{\leq}{\leqslant}
\newcommand{\<}{\left\langle}
\renewcommand{\>}{\right\rangle}
\newcommand{\8}{\infty}

\renewcommand{\:}{\colon}
\renewcommand{\a}{\alpha}

\newcommand{\Aut}{\mathrm{Aut}\,}
\renewcommand{\b}{\beta}

\newcommand{\cs}[1]{\mathrm{Set}_{#1}}
\renewcommand{\Cap}[2]{\underset{#1}{\overset{#2}{\bigcap} }}

\renewcommand{\Cup}[2]{\underset{#1}{\overset{#2}{\bigcup} }}

\newcommand{\D}{\mathfrak{D}}
\newcommand{\e}{\epsilon}

\newcommand{\f}{\varphi}
\newcommand{\Fl}{\mathbb{F}}
\newcommand{\fix}{\mathrm{Fix}}
\newcommand{\F}{\mathbb {F}}
\newcommand{\g}{\gamma}
\newcommand{\G}{\Gamma}

\newcommand{\Homeo}{\mathrm{Homeo}}
\newcommand{\hood}{\mathcal{N}}

\newcommand{\Inf}[1]{\underset{#1}{\inf}}
\newcommand{\Inn}{\mathrm{Inn}\,}
\renewcommand{\int}{\varint}
\newcommand{\Isom}{\mathrm{Isom}\,}

\renewcommand{\L}{\Lambda}

\newcommand{\Lim}[1]{\underset{#1}{\lim}}

\renewcommand{\max}[1]{\underset{#1}{\mathrm{max}}}

\newcommand{\N}{\mathbb{N}}

\renewcommand{\O}{\mathcal{O}}

\newcommand{\Out}{\mathrm{Out}}
\renewcommand{\P}{\mathbb{P}}

\newcommand{\Prod}[2]{\underset{#1}{\overset{#2}{\prod} }}

\newcommand{\PSL}{\mathrm{PSL}}
\newcommand{\PGL}{\mathrm{PGL}}
\newcommand{\Q}{\mathbb{Q}}

\newcommand{\R}{\mathbb{R}}

\newcommand{\set}{\mathrm{Set}}
\newcommand{\SL}{\mathrm{SL}}

\newcommand{\stab}{\mathrm{Stab}}
\newcommand{\Sum}[2]{\underset{#1}{\overset{#2}{\sum} }}
\newcommand{\Sup}[1]{\underset{#1}{\sup}\,}
\newcommand{\Sym}{\mathrm{Sym}}

\newcommand{\X}{\mathbb{X}}
\newcommand{\mfN}{N}
\newcommand{\nl}{\mathrm{(NL)}}

\author{Talia Fern\'os, Sahana H Balasubramanya}

\address{Department of Mathematical Sciences, Lafayette College, Pardee Hall, USA}
\email{hassanba@lafayette.edu}

\address{Department of Mathematics, Vanderbilt University, 1326 Stevenson Center, Nashville, TN, USA}
\email{Talia.Fernos@Vanderbilt.edu}

\begin{document}
\renewcommand\windowpagestuff{\rule{2cm}{2cm}}
\begin{abstract}
We present a new notion of non-positively curved groups: the collection of discrete countable groups acting (AU-)acylindrically on finite products of $\delta$-hyperbolic spaces with general type factors and associated subdirect products. This work is inspired by the classical theory of $S$-arithmetic lattices and the flourishing theory of acylindrically hyperbolic groups.  In this paper - the first of three - we develop various fundamental results, explore elementary subgroups in higher rank, and exhibit a free vs abelian Tits Alternative. Along the way we give  representation-theoretic proofs of various results about acylindricity -- some methods are new even in the rank-one setting.    
 
The vastness of this class of groups is exhibited by recognizing that it contains $S$-arithmetic lattices with rank-one factors, acylindrically hyperbolic groups, colorable HHGs, groups with property (QT), and enjoys robust stability properties. 
\end{abstract}

\maketitle 

\begin{center}
\emph{In honor of Bhama Srinivasan\\ and Alex Furman and Denis Osin \\ on the occasions of their recent multiples of 10$^\textit{th}$ birthdays}
\end{center}
\begin{figure}[h]
    \centering
    \includegraphics[width=.58\linewidth]{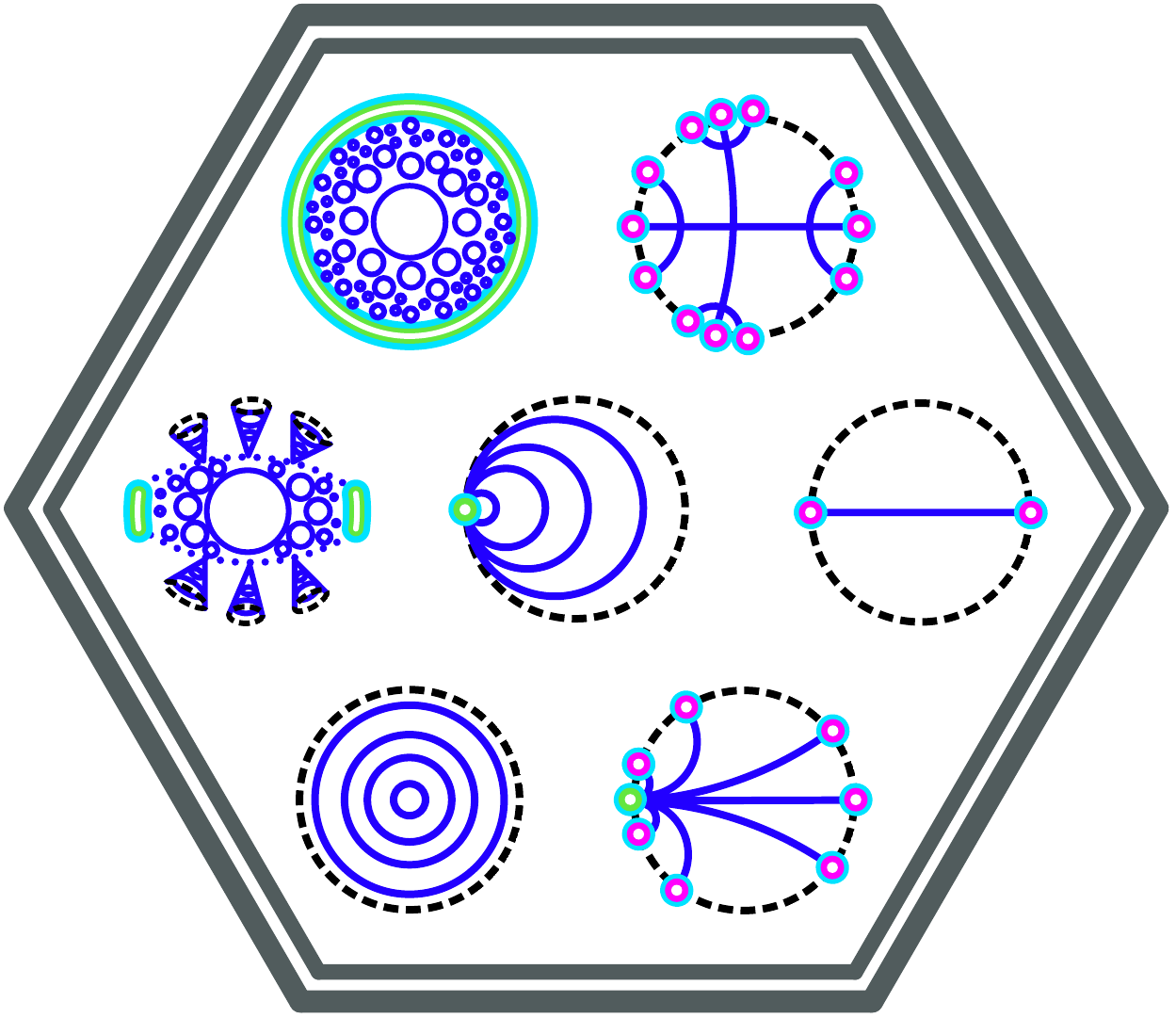}
    
    \emph{Figure showing the classification of actions on a $\delta$-hyperbolic space}
       
    Counter-clockwise from 3 o'clock: lineal, general type, tremble, rift, rotation, quasi-parabolic 

    Center: parabolic 
\end{figure}

\thispagestyle{empty}

\newpage
\tableofcontents
\thispagestyle{empty}

\newpage 
\addtocounter{page}{-2}
\section*{Table of Standing Notation}
\begingroup
\renewcommand*{\arraystretch}{1.75}
\begin{center}
\begin{table}[H]
\begin{tabular}{|c|c|}
 \hline 
 $\G$ & discrete countable group 
 \\
  \hline 
 $H$ & arbitrary subgroup 
 \\
 \hline
   $X$ & complete separable $\delta$-hyperbolic   \\
   or  $X_i$ & geodesic metric space \\
 \hline
  $Y, Z$ & arbitrary metric space 
  \\
\hline
$\hood_\e(\cdot)$, $\~\hood_\e(\cdot)$  & open, closed $\e$-neighborhood   \\
\hline
$\set_\e A$ & pointwise $\e$-coarse stabilizer of the set $A$
\\ 
\hline
$\stab_\G A$ & (setwise) stabilizer of the set $A$ in $\G$ \\
\hline 
$\fix_\G A$ & pointwise fixator of the set $A$ in $\G$ \\
 \hline  $R=R(\e), N=N(\e)$ & acylindricity constants, depend on $\e$
 \\\hline
 $\X= \Prod{i=1}{\D}X_i$ & product  of $\delta$-hyperbolic spaces 
 \\\hline
 $x\in\X, x=(x_1, \dots,x_i,\dots x_\D)$ & notation for coordinate factors
 \\
 \hline
 $\D$ & number of factors in product
 \\\hline 
 $\~X=\partial X\cup X$ & the Gromov bordification
 \\\hline $\~\X = \Prod{i=1}{\D} \~X_i$ & bordification of $\X$\\
 \hline 
 $\partial \X = \~\X \setminus \X$ & boundary of $\X$
 \\\hline $\partial_{reg}\X = \Prod{i=1}{\D} \partial X_i$ & regular boundary of $\X$\\
 \hline 
 $\Sym_\X(\D)$ & permutation group of isometric factors\\
 \hline 
 $\Aut \X =  \Sym_\X(\D)  \ltimes \Prod{i=1}{\D} \Isom X_i$ & automorphism group
 \\
 \hline
  $\Aut_{\!0} \X = \Prod{i=1}{\D} \Isom X_i$ & factor preserving automorphism group 
  \\\hline
  $g\in \Aut_{\!0} \X$, $g=(g_1, \dots, g_i, \dots, g_\D)$ & notation for coordinate factors
  \\
  \hline $\Isom_{\! 0} L$ & kernel of action on $\partial L$, $L$ a quasiline\\
  \hline 
  $\G_0\leq \G$ & (usually) the subgroup that maps to $\Aut_{\!0}\X$ \\
  \hline 
  $\mathfrak C_\G(g), \mathfrak N_\G(H)$ &   centralizer, normalizer in $\G$, for $g\in \G$, $H\leq \G$  \\
  \hline 
\end{tabular}
\caption{}
\label{table:notation_p1}
\end{table}
\end{center}
\endgroup

\section{Introduction}

Taking inspiration from the work of Sela for groups acting on trees \cite{SelaAcylAcces}, in 2008 Bowditch introduced the study of acylindrical actions  on general $\delta$-hyperbolic spaces with an eye towards understanding the action of (surface) mapping class groups on their respective hyperbolic curve complexes \cite{MasurMinsky, Bowditch2008}.

 In 2015, Osin introduced the class of acylindrically hyperbolic groups (i.e. those that admit a general type acylindrical action on a $\delta$-hyperbolic space), and proved that this succinct condition unifies many other classes which had previously been studied \cite{Acylhyp}. Philosophically, one may view the notion of an acylindrically hyperbolic group as a generalization of a uniform lattice in the isometry group of a locally compact $\delta$-hyperbolic space.  Since then, many implications of acylindrical hyperbolicity  have been established, guided by this philosophy  but requiring new techniques that circumvent the strength of properness and coboundedness conferred by  a geometric action (see \cite{surveyosin} for a survey). Because the defining condition is so succinct, the class of acylindrically hyperbolic groups  is vast and includes the groups classically studied in geometric group theory with hyperbolic-like properties: e.g.  (irreducible) right angled Artin groups,  $\Out(F_n)$, and many hierarchically hyperbolic groups (The class of hierarchically hyperbolic groups is another example of a class that unifies various others \cite{BHS1}). Noteworthy  members of this class are mapping class groups as well as all lattices (uniform or not) in the isometry groups of locally compact hyperbolic spaces, e.g. simple rank-one Lie groups.

In 1979 Harvey asked whether mapping class groups were arithmetic lattices \cite{Harvey}. Although this was answered in the negative by Ivanov \cite{Ivanov}, there are many philosophical parallels between mapping class groups and lattices (e.g. the finiteness of their outer-automorphism groups \cite{Ivanov1984, McCarthy, Gundogan}). Our work here is the consequence of extending nonelliptic-acylindrical actions as a proxy for geometric actions to the higher-rank setting. Specifically, we shall consider finite products of (general type) $\delta$-hyperbolic spaces with the $\ell^2$-product metric as \emph{nonpositively curved spaces} as the playground and groups admitting acylindrical-like actions as the players. With the aim of creating a framework that encompasses acylindrically hyperbolic groups as well as $S$-arithmetic lattices, uniform or not, we allow for non-uniform acylindricity as well. The two concepts are united under the term of \emph{acylindricity of ambiguous uniformity} (which we refer to as AU-acylindricity, see Definition~\ref{defn:typesofactions_p1}).  This landscape is large as is witnessed by the fact that it contains the already large class of acylindrically hyperbolic groups,  $S$-arithmetic lattices in semi-simple Lie groups with rank-one factors and enjoys robust stability properties. 
In this first of three papers, we develop a rigorous framework in which to place this playground and establish necessary fundamental results, such as a Tits Alternative.

In this Part I, we use the existing results of acylindrically hyperbolic groups, i.e. acylindricity in rank-one, to guide our development of the semi-simple theory of AU-acylindricity in higher rank. As Section~\ref{sec:mainthms} shows, we indeed extend many of the rank-one results. Consequently, this paper also serves as a starting point for systematically studying (AU-) acylindrical actions on finite products of $\delta$-hyperbolic spaces, a study that necessitated revisiting proofs in rank-one, and a natural trifurcation of elliptic actions. Understanding the nature of the elliptic actions will also allow us to ``tame" the actions in Part II \cite{BFPart2}. This will turn out to be a key step in establishing  a type of semi-simplicity in our higher rank setting, concretely bridging our two sources of inspiration: acylindrically hyperbolic groups and $S$-arithmetic lattices in products of rank-one groups.

 Part II \cite{BFPart2} also considers semi-simple-type consequences, such as the existence of a ``strongly canonical" product decomposition and its ramification for  a recent conjecture of Sela  announced during his Aisenstadt Chair Distinguished Lecture Series in June 2023 at the Centre de Recherches Math\'emathiques in Montreal overlapping with the \emph{Groups Around 3-Manifolds} conference, both of which were part of the Thematic Program in Geometric Group Theory. In Part III, we shall extend results from \cite{AbbottManning} to the higher-rank setting and examine a type of irreducibility witnessed on the level of maximal flats \cite{BFPart3}.

It is worth noting that  within the framework of AU-acylindrical 
actions on hyperbolic spaces, elliptic and parabolic actions are not capable of yielding meaningful information about the acting group. Indeed, given an isometric action of a group $\G$ on a metric space $Y$ one can create new ($\delta$-hyperbolic) spaces $Y_{ell}$, and $Y_{par}$ on which $\G$ acts by isometries. To obtain  $Y_{ell}$, one may add an additional ``central" point and attach a length-1 edge from that point to all other points in $Y$, yielding a space of diameter 2. The corresponding action of $\G$ has a fixed point and is acylindrical. In the case of $Y_{par}$, one may perform a horoball construction, as in  \cite{GrovesManning} and obtain a level-preserving, i.e. parabolic action of $\G$ on $Y_{par}$. Both such constructions yield an AU-acylindrical action on a $\delta$-hyperbolic space. Since this applies to any isometric action, clearly no meaningful information can be gleaned from such a framework.  This is one reason why we eventually focus on the ``semi-simple" case, i.e. the situation where the factor actions are all of general type, though we do not always restrict the factor actions in this way.

When all factors of $\X$ admit general type actions, one may  view it as a non-elementary analogue of a vector space. In this sense, the theory we develop of AU-acylindrical actions on finite products of $\delta$-hyperbolic spaces can be seen from a representation-theoretic point of view. In particular,  the collection of \emph{all} such actions can provide meaningful information about the group, a truth that goes back to the classical theory of linear representations and is also reflected in the results from \cite{ABO}. This is in contrast with the trend of considering geometric actions as ``best". However, mapping class groups are both acylindrically hyperbolic and have Property (QT) i.e. act by isometries on a finite product of quasi-trees so that the orbit maps are quasi-isometric embeddings. Therefore they admit a general type acylindrical action on a single hyperbolic space, as well as a proper and hence AU-acylindrical action on a product of quasi-trees.  Because each of these actions bring to light different and interesting properties, their relative value is incomparable. We also note that while elliptic actions can not provide meaningful information from the point of view of  ``nonpositive curvature", the class of such actions for a group $\G$  (via the above construction)  encompasses the entire universe of its isometric actions, including such important classes as its isometric actions on Hilbert or Banach spaces, in particular the left-regular or quasi-left-regular representations.

A point of interest to the geometric group theorist may be the fact that $\X$, under mild hypotheses on the factors, has a highly structured quasi-isometry group (e.g. the analogue to our $\Aut \X$ for quasi-isometries; which is virtually the product of the isometry groups of the factors, see Table~\ref{table:notation_p1}, or Section~\ref{Subsec:Products}) \cite{KleinerLeeb,EskinFarb,KapovichKleinerLeeb,Bowditch2016}. While quasi-isometry is the leading type of equivalence within  geometric group theory, particularly in the context of hyperbolicity, we remark that there are many interesting and important notions that are not invariant under quasi-isometry, for example the geometric CAT(0) property or the rigid Property (T) for groups (see \cite[Theorem 3.6.5]{BdlHV}). Relevant to our work is the fact that cocompact lattices in the same product of ambient groups are necessarily quasi-isometric, independent of whether they are reducible or irreducible. Therefore finitely presented simple lattices in products of trees  \cite{BurgerMozes2000, Wise} are quasi-isometric to reducible cocompact lattices in the same, which are virtually products of free groups. Therefore, while the tension between hyperbolicity and quasi-isometry has been incredibly fruitful, when we enter the world of non-positive curvature in higher-rank,  we must allow ourselves to let go of quasi-isometry as the sole focus.

 Since we are considering $\X$ with $d_2$, the $\ell^2$-product metric as nonpositively curved, one may wonder how far our designated group $\Aut \X$ is from the isometry group $\Isom(\X, d_2)$. Under natural hypotheses, these are the same.  We prove this using  \cite{FoertschLytchak}  in Part II \cite{BFPart2}.  We note that CAT(0) cube complexes effectively have the same structure \cite{Bregman2017} (see also \cite{CapraceSageev}).   

We note that we are not the first to consider variations on acylindricity: there are those studied by Hamenst\"adt \cite{Hamenstaedt}, Delzant \cite{Delzant}, and Genevois \cite{Genevois} (see also \cite{ChatterjiMartin}), effectively in the hyperbolic setting, as well as
Sela \cite{SelaHR1Pub, SelaHR2}  and  Wan and Yang
\cite{WanYang} in higher rank (see also \cite{Button}). One can easily verify that our nonuniform acylindricity implies the version studied by Hamenst\"adt. The weak acylindricity considered by Delzant is a uniform version of the WWPD property introduced by Bestvina-Bromberg-Fujiwara \cite{BBF}. Genevois' nonuniform acylindricity and weak acylindricity is respectively our AU-acylindricity and the standard acylindricity condition with $\e=0$.  Delzant's weak acylindricity also inspired the higher rank work of both Sela and Wan-Yang. The latter studies proper actions on products of hyperbolic spaces with factor actions that are weakly acylindrical in the sense of Delzant. Sela's versions of acylindricity concern his recent work on higher rank JSJ decompositions and we will consider these more deeply in Part II \cite{BFPart2} (see also \cite[Section 10]{BF}).


\section{Main Theorems}\label{sec:mainthms}

As Table~\ref{table:notation_p1} summarizes, we will use $\X$ to denote a finite product of $\D$-many complete, separable, $\delta$-hyperbolic geodesic spaces (which will often have isometry groups whose actions are of general type). Our target group for representations is  $\Aut \X$, which consists of  permutations of isometric factors along with the product of the isometry groups of the factors. 

\begin{defn} Let $\G$ be a discrete countable group, $Y$ a metric space and  $\G\to \Isom Y$ an isometric action.
\begin{itemize}
    \item If $\e>0$ then the joint $\e$-coarse stabilizer of points $x,y\in Y$ will be denoted by $$\set_\e\{x,y\} = \{g\in \G: d(gx,x), d(gy,y)\leq \e\}.$$
    \item The action is called \emph{AU-acylindrical}, or \emph{acylindrical of ambiguous uniformity} if for every $\e>0$, there exists  $R\geq 0$ such that for any points $x,y \in Y$ with $d(x,y) \geq R$, we have an ambiguous bound on cardinality $| \set_\e\{x,y\}| <\8$. 
    \item  The action is called \emph{acylindrical} if for every $\e>0$, there exists  $R\geq 0, N \geq 0$ such that for any points $x,y \in Y$ with $d(x,y) \geq R$, we have the uniform bound on cardinality $| \set_\e\{x,y\}| < N$. 
     \item  The action is called \emph{nonuniformly acylindrical} if it is AU-acylindrical but not acylindrical.
\end{itemize}    
\end{defn}

Groups that admit a non-elliptic acylindrical action in rank-one satisfy a Tits Alternative: they are either virtually  $\Z$ or contains a nonabelian free group \cite[Theorem 1.1]{Acylhyp}. 
Let $H$ be a group. We say a group $G$ is \emph{virtually} $H$ if $G$ contains a finite index subgroup that is isomorphic to $H$.
We contribute to the landscape of Tits Alternatives with:

\begin{thmA}[Tits Alternative]\label{intro:titsalt}
    Let $\G\to \Aut\X$ be acylindrical and not elliptic. Either $\G$ contains a nonabelian free group or $\G$ is virtually $\Z^k$, where $1\leq k\leq \D$, where $\D$ is the number of factors of $\X$.
\end{thmA}

Closely linked to the proof of Theorem~\ref{intro:titsalt} is the structure of the stabilizers of elements in the regular boundary of $\X$. In this vein, we prove the following result, which says, roughly speaking, that in an acylindrical action on $\X$, fixed points on the regular boundary must occur in pairs. Recall that $\partial_{reg} \X = \Prod{i=1}{\D} \partial X_i$ is the  regular boundary.

\begin{thmA}\label{intro:stabs}
Suppose that $\xi \in \partial_{reg}  \X$, where $\G \to \Aut\X$ is acylindrical. If $\fix_\G \{\xi\}$ is infinite then there exists an $\xi' \in \partial_{reg} \X$ such that  $\xi\neq \xi'$ and yet $\fix_\G\{\xi\} = \fix_\G\{\xi, \xi'\}$. Moreover, if $\D'= |\{i: \xi_i\neq \xi_i'\}|$ then $\fix_\G \{\xi\}$ is virtually   $\Z^k$ for some $1 \leq k \leq \D'\leq \D$, where $\D$ is the number of factors of $\X$.
\end{thmA}

  To emphasize that acylindricity extends the notion of cocompact lattices, we note that Godement's Compactness Criterion states that  $S$-arithmetic lattices are cocompact if and only if they contain no nontrivial unipotents \cite{BHC, Behr} (see Theorem~\ref{thm: Godement}). Nontrivial unipotent subgroups are precisely those that have a distinct fixed point on the regular boundary in the (irreducible) $S$-arithmetic case. In particular, any solvable subgroup is virtually diagonalizable, i.e. stabilizes a flat in the corresponding product of symmetric spaces and Bruhat-Tits buildings (which is a CAT(0) space when endowed with the $\ell^2$-product metric). Therefore, such a solvable subgroup acts properly on said flat and is hence $\Z^k$, where $k$ is bounded by the rank of the CAT(0) space. The classical Tits Alternative \cite{Tits} states that a linear group either contains a free group or is virtually solvable, and hence we recover a Tits Alternative as above for subgroups of such lattices (see Section \ref{sec: Godement}).

We also consider \emph{regular} elements, the higher rank analogue of loxodromic elements, namely those that do not permute the factors and whose projection to each factor is loxodromic (classically a regular element refers to one whose axis is contained in a unique maximal flat). Given $\G \to  Aut\X$  and $\g\in \G$ regular, the associated \emph{elementary subgroup} is $E_\G(\g):= E_\G(\g^-, \g^+):= \Cap{i=1}{\D}\fix_\G\{\g_i^-, \g_i^+\}$,  where $\g^-, \g^+\in \partial_{reg}\X$ are respectively the repelling and attracting fixed points for $\g$ (see Definition \ref{def: elem subgroup}).

Maher and Tiozzo used random walks to show that loxodromic elements in a general type action can be encountered asymptotically almost surely \cite{MaherTiozzo}. We use their result to easily conclude that regular elements exist. The structure of the centralizer is a higher-rank analog of \cite[Proposition 6]{BestvinaFujiwara}.

\begin{thmA}[Regular elements]\label{intro:regelts} Let $\G\to \Aut \X$ be an action with general type factors. Then there exists $\g\in \G$ that acts as a regular element. Furthermore, if the action is AU-acylindrical and $\g\in \G$ is a regular element then the centralizer $\mathfrak C_\G(\g)$ and the elementary subgroup $E_\G(\g)$ are both virtually $\Z^k$, for the same $k$ such that $1\leq k\leq \D$.
\end{thmA}

Another well known result in rank-one is that pertaining to the classification of isometries and isometric actions (see Section~\ref{sec:isomrank1} and Theorem~\ref{thm:actionsclassrk1}). In this paper, we observe  that the well known classification of isometries and actions in rank-one can be refined:  elliptic elements and actions admit a trichotomy, manifesting themselves as \emph{trembles, rotations or rifts} (see Definition~\ref{defn:elementsclassextn}, and the left-most three images in the figure on the title page). Of note, the trembles shall in a sense take on the role of the center in a semi-simple lattice. This analogy is further extended to a  semi-simple dictionary in Part II \cite{BFPart2}, where controlling the trembles with the use of acylindricity will be helpful in dealing with results and proofs pertaining to lattice envelopes.

 Osin demonstrated that there is a tension between the acylindricity hypothesis and the possibilities for isometric actions on $\delta$-hyperbolic spaces -- only elliptic, lineal or general type actions are compatible (i.e. parabolic and quasiparabolic acylindrical actions do not exist see \cite[Theorem 1]{Acylhyp}). This implies that a group acting acylindrically will not contain elements acting parabolically.  We extend this result to the higher rank setting by proving the following, which furthers the analogy with ($S$-arithmetic) semi-simple cocompact lattices via Godement's Compactness Criterion  (see Theorem~\ref{thm: Godement}).

\begin{thmA}[Obstructions to acylindricity in higher rank]\label{intro:acylprob}Let $\G \to \Aut_{\!0}\X$  such that the projections $\G\to \Isom X_i$ are all either elliptic, parabolic or quasiparabolic (with at least one factor being parabolic or quasiparabolic). Then $\G \to \Aut_{\!0} \X$ is not acylindrical.
In particular, if $\G \to \Aut_{\!0}\X$ is acylindrical, then every element of $\G$ is either an elliptic isometry or contains a loxodromic factor.
\end{thmA}

As a consequence of this obstruction, we are able to deduce (see Corollary~\ref{cor:fin gen} and Proposition~\ref{prop:SL2ZS}) that while groups such as $\PSL_2\Z[\frac{1}{p}]$ act (with general type factors)  nonuniformly properly, and hence nonuniformly acylindrically,  on the product of $\delta$-hyperbolic spaces, it does not act acylindrically on any $\X$, unless the action is elliptic. This is in contrast to the rank-one setting. For example, $\PSL_2\Z[i]$ is a nonuniform lattice in $\PSL_2\mathbb C$ and hence has a general type nonuniformly proper action on the associated symmetric space, which is the (real) hyperbolic 3-space. However,  thanks to the existence of a WPD element and the associated  construction of Bestvina, Bromberg, and Fujiwara \cite{BBF, BBFS}, $\PSL_2\Z[i]$  also admits a general type acylindrical action on some other  $\delta$-hyperbolic space (which is necessarily not locally compact).

Using  the superrigidity for actions on hyperbolic spaces established by Bader, Caprace, Furman, and Sisto \cite{BaderCapraceFurmanSisto} (see Theorem~\ref{Thm: BCFS}), and representation theoretic tools we prove the following theorem. The reduction in the statement means passing to the essential cores of the factors of $\X$, as in Proposition~\ref{prop: ess core}, and the equivalence for actions  introduced in \cite[Definition 4.2]{BaderCapraceFurmanSisto}.

\begin{thmA}\label{intro: HR lattice AUacyl} 
 Fix integers ${\mathbf{D}}\geq 0$ and $F\geq 0$. Consider an irreducible lattice $\G\leq G$  where $G= \Prod{i=1}{\mathbf{D}}H_i\times \Prod{j=1}{F}G_j$ is a product of centerless simple Lie groups, $H_1, \dots, H_{\mathbf{D}}$ have real rank one, and $G_1, \dots, G_F$ have real rank strictly larger than one. 
Let $S_1, \dots, S_{\mathbf{D}}$ be the hyperbolic symmetric spaces for $H_1, \dots, H_{\mathbf{D}}$. There exists an AU-acylindrical action $\G\to \Aut\X$ with general type factors if and only if $F=0$, and up to a natural reduction, $\X =\mathbb S\times \mathbb S_{\mathrm{dup}}$, where $\mathbb S= \Prod{i=1}{\mathbf{D}}S_i$ and $\mathbb S_{\mathrm{dup}}$ is the (possibly empty) product of duplicate copies of factors of $\mathbb S$. Furthermore, the action is acylindrical if and only if $\G\leq G$ is cocompact.
\end{thmA}

Earlier, we highlighted that nonuniform lattices in the isometry groups of locally compact hyperbolic spaces are acylindrically hyperbolic. A more general fact is true: in rank-one a nonuniformly acylindrical general type action can be used to find an acylindrical general type action (via \cite{BBFS}).  Theorem~\ref{intro: HR lattice AUacyl}, as well as the examples in  Section~\ref{sec: SL2} show that this is no longer the case in higher rank. Namely, these irreducible nonuniform lattices \emph{do not} admit acylindrical actions on any product of $\delta$-hyperbolic spaces. 
    
\begin{question}
    Which groups admit a nonuniformly acylindrical action on a finite product of $\delta$-hyperbolic spaces but do not admit an nonelliptic acylindrical action on any such finite product?
\end{question}

\noindent
\textbf{A note on \cite{BF}:} Much of the work presented here is available on the arXiv \cite{BF}, where we claimed \underline{\smash{without proof}} that the class of groups admitting an AU-acylindrical action on a product of hyperbolic spaces was closed under direct products. In fact, this is not true and only holds when the actions are proper. Similarly, Petyt and Spriano claimed without proof in \cite[Remark 4.9]{PetytSpriano} that the action of an HHG on the product of its eyries (i.e. the product of the maximally unbounded domains) is acylindrical which is not true for the same reasons. In a different direction, Bader, Caprace, Furman, and Sisto retracted a claim concerning rigidity for actions of $S$-arithmetic lattices on hyperbolic spaces, and must rather assume the ambient group is Lie \cite{BaderCapraceFurmanSisto} (see Theorem~\ref{Thm: BCFS} for a precise statement). 

However, we are thankful for how these events unfolded as it led us to explore the existence of a strongly canonical product decomposition  (see \cite[Theorem A]{BF}) which in turn allowed us to address the recent conjecture of Sela mentioned above (see \cite[Section 10]{BF}).  We perhaps may not have explored these directions had we remembered that the class is not closed under products, or potentially believed that $S$-arithmetic lattices were somehow exceptional (see Section~\ref{sec: SL2} for the case of $\SL_2 \mathfrak R$).

We note that despite the above issues, the  results \underline{\smash{proven}}  in \cite{BF} are true. In fact, the lack of closure under direct products strengthens the canonical product decomposition.  We also note that Sela's Conjecture concerns colorable HHG's. According to \cite[Theorem 3.1]{HagenPetyt}, these act properly and therefore AU-acylindrically on a finite product of $\delta$-hyperbolic spaces and so do admit a strongly canonical product decomposition. These claims will be  shown in the forthcoming Part II \cite{BFPart2}, where we also strengthen our treatment of the conjecture. 

 We are also thankful for the timing of Sela's Conjecture. Our first attempt at addressing it led us to consider subdirect products more carefully. The framework of subdirect products was then a solution we had on-hand to address the lack of closure discussed above. In Part II, we establish the stability of our results for (subdirect) products.

 Addressing the above issues required several additional pages of writing which led to the  splitting of the work from \cite{BF} into the present Part I and the  forthcoming Part II \cite{BFPart2} for the convenience of the reader.
\hfill $\blacksquare$

\medskip 
\noindent
\textbf{Acknowledgements: } The authors would like to thank Uri Bader, Jennifer Beck, Yves Benoist, Corey Bregman, Nic Brody, Montserrat Casals-Ruiz, Indira Chatterji, R\'emi Coulon, Tullia Dymarz, David Fisher, Alex Furman, Daniel Groves, Thomas Haettel, Jingyin Huang, Michael Hull, Kasia Jankiewicz, Fra\c cois Labourie, Marco Linton, Robbie Lyman, Joseph Maher, Jason Manning, Dan Margalit, Mahan Mj, Shahar Mozes, Denis Osin,  Harry Petyt, Jacob Russell, Michah Sageev, Zlil Sela, Alessandro Sisto, Davide Spriano, Henry Wilton, Daniel Woodhouse, and Abdul Zalloum  for useful conversations. 

 The first named author would like to thank Lafayette College for its generous start-up grant. The second named author would like to thank  the American Institute of Mathematics, the Fields Institute, the Insitut Henri Poincar\'e, The University of North Carolina, Greensboro, and the National Science Foundation for the support through NSF grants DMS 1802448 and 2005640.

Commutative diagrams were (eventually) made with the use of quivr by varkor, which can currently be found at (\url{q.uiver.app/}) with public github repository valkor/quivr (\url{https://github.com/varkor/quiver}).

\medskip 
\noindent \textbf{Structure of the paper: } Section~\ref{sec:basics} deals with some basic notation, definitions and results that will be used in this paper and subsequent ones. These are divided into 3 subsections, on products, hyperbolicity, and the bordification of a hyperbolic space. Sections~\ref{sec:isomrank1} and~\ref{sec:isomhigherrank_p1} address the classification of actions and isometries in rank-one and higher rank respectively; the first part of Theorem~\ref{intro:regelts} is proved as Proposition~\ref{Prop:Regular Exist}, and the second part is proved in Corollary~\ref{cor: cent reg elmt}.  Examples and non-examples are discussed in detail in Section~\ref{sec:examples_p1}, where Theorem~\ref{intro: HR lattice AUacyl}  is proved as Theorem~\ref{thm: HR acyl iff cocompact}. Section~\ref{sec:elemsubgrp} establishes the structure of elementary subgroups and the Tits Alternative Theorem~\ref{intro:titsalt}. Theorem~\ref{intro:stabs} is also proved therein as Propositions~\ref{prop:elemsubvirtab} and~\ref{prop: reg fixed implies reg pair}.


\section{Metric Basics}\label{sec:basics}

We begin with notions that do not require any special properties about the metric space, before progressing to complete separable $\delta$-hyperbolic geodesic spaces and their finite products. We collect some definitions and facts that we will need in this section. Since there are no geometric requirements on these results we state them for general metric spaces.

\subsection{The Polish Topology on $\Isom Y$} \label{Sec:toponisomgroup} Recall that a group is said to be \emph{Polish} if it is separable and admits a  complete metric with respect to which the group operations are continuous maps. Polish groups fit in a hierarchy as in Figure~\ref{fig:discretetopolish_p1}. 

\begin{figure}
    \centering
    \includegraphics[width=0.55\linewidth]{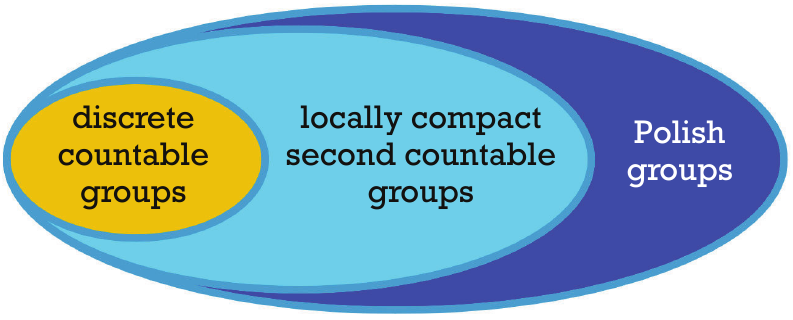}
    \caption{From discrete to Polish}
    \label{fig:discretetopolish_p1}
\end{figure}

If $Y$ is a complete separable metric space then  $\Isom Y$ is a Polish group with respect to the topology of pointwise convergence. More specifically, if $\{x_n\}\subset Y$ is a countable dense set and $g,h\in \Isom Y$ then the following metric is a complete  metric on $\Isom Y$ that generates the topology of pointwise convergence \cite[Example 9.B(9)]{Kechris}:

\begin{equation}\label{Polish_Metric_p1}
    d(g,h) = \Sum{k=0}{\8} \frac{1}{2^{n+1}}\left(\frac{d(g(x_k), h(x_k))}{1+d(g(x_k), h(x_k))} + \frac{d(g^{-1}(x_k), h^{-1}(x_k))}{1+d(g^{-1}(x_k), h^{-1}(x_k))}\right).
\end{equation}

If $Y$ is locally compact, separable, connected (which is the case when $Y$ is geodesic), and complete then by \cite{ManoussosStrantzalos} we have that $\Isom Y$ is locally compact with respect to the above metric and acts properly on $Y$. If, in addition, $Y$ is geodesic then every closed metric ball is compact by the Hopf-Rinow Theorem \cite[Proposition I.3.7]{BridsonHaefliger}. It is for these reasons that we will have the standing assumption that our metric spaces are separable and complete.   

\begin{defn}
    $Y$ is said to be uniformly locally compact if for every $\rho, \tau>0$ there is a $C(\rho, \tau)> 0$ such that for every $x\in Y$ and every open cover of the closed ball $\~\hood_\rho(x)$ by open $\tau$ balls, there is a subcover of cardinality at most $C(\rho, \tau)$.
\end{defn}

We note that if a metric space has cocompact isometry group then it is uniformly locally compact.

\subsection{Actions and AU-acylindricity}

Let $S$ be a set on which the group $\G$ acts, and let $A\subset S$. We shall denote the stabilizer of $A$ by $\stab A$ which is the set of $g\in \G$ such that  $g A=A$, and the fixator $\fix A = \Cap{a\in A}{} \stab\{a\}$. Note that for a single element $\fix\{a\} = \stab\{a\}$, while more generally $\fix A\leq \stab A$. This last subgroup inclusion is of finite index when $A$ is finite.

\begin{defn}\label{def:coarse_stab_p1}
    For an action $\G\to \Isom Y$, $x\in Y$, and $\e>0$, the $\e$\emph{-coarse stabilizer} of $x$ is $\set_\e\{x\}:=\{g\in \G: d(gx,x)\leq \e\}$. If $S\subset Y$ then we denote by $\set_\e S= \Cap{x\in S}{} \set_\e\{x\}$. The notation $\hood_\e(x)$ denotes the $\e$-neighborhood of $x$ in $Y$ (with respect to the metric on $Y$).
\end{defn}

\begin{defn}\label{defn:typesofactions_p1} An action  $\G\to \Isom Y$ is said to be 
\begin{enumerate}
\item \emph{cobounded} if for all $x\in Y$ there is an $R>0$ such that $\G\cdot\hood_R(x)=Y$;
\item \emph{cocompact} if there exists $C\subset Y$ compact such that $\G\cdot C= Y$;
 \item \emph{proper} if for every $\e>0$ and every $x\in Y$ we have $|\set_\e\{x\}|<\8$;
 \item \emph{uniformly proper} if for every $\e>0$ there is an $N>0$ so that for every $x\in Y$ we have $|\set_\e\{x\}|\leq N$;
    \item \emph{acylindrical}, if for every $\e>0$, there exist nonnegative constants $R=R(\e) $ and $ N=N(\e)$ such that for every $x,y \in X$ with $d(x,y) \geq R$, we have $$| \set_\e\{x,y\} |\leq N;$$  
    \item \emph{AU-acylindrical}, or \emph{acylindrical of ambiguous uniformity} if for every $\e>0$, there exists  $R\geq 0$ such that for every $x,y \in Y$ with $d(x,y) \geq R$, we have $| \set_\e\{x,y\}| <\8;$
    \item \emph{nonuniformly acylindrical}, if it is AU-acylindrical but not acylindrical. 
\end{enumerate}

\end{defn}

 The reader may think of AU-acylindricity as a type of properness for the associated action on $Y \times Y$, where $\e$-properness is only guaranteed on the complement of a ``thick diagonal", i.e. outside of the set of pairs that are at distance at most $R$ in $Y$. 

We also note that the type of action a lattice enjoys on its ambient space is always proper, and that proper actions are always AU-acylindrical. Acylindricity generalizes uniformly proper actions, in particular those of cocompact lattices. Tautologically, nonuniform acylindricity generalizes those that are not uniform, in particular the action of a nonuniform lattice on its ambient space. This is summarized in Figure~\ref{fig: LCSC to acyl}.

\begin{figure}[H]
    \centering
    \includegraphics[width=0.55\linewidth]{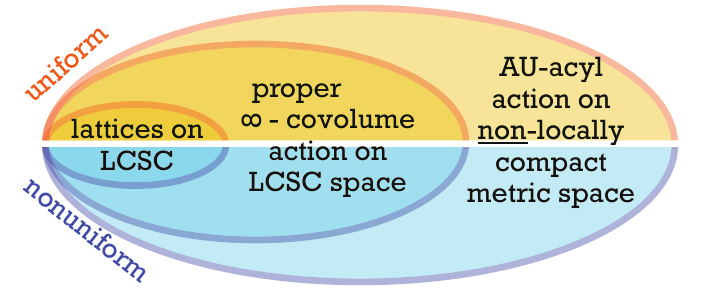}
    \caption{From lattices to AU-acylindrical actions }
    \label{fig: LCSC to acyl}
\end{figure}

We now state some elementary lemmas that will be used in future sections.

\begin{lemma}\label{Lem:prop cobound is acyl}
    If an action  $\G\to \Isom Y$ is  proper and cobounded then it is acylindrical. 
\end{lemma}

Since elements in the kernel of an action must fix all points in the space, we immediately have the following: 

\begin{lemma}\label{Lem:AU-acyl has finite kernel}
    If $\rho:\G \to \Isom Y$ is AU-acylindrical and $Y$ is unbounded then the kernel $\ker(\rho)= \{g\in \G: \rho(g) = \mathrm{id}_Y\}$ is finite.
\end{lemma}

We now prove that the notion of AU-acylindricity coincides with the notion of properness for actions on locally compact spaces that are ``sufficiently continuously geodesic" (though AU-acylindricity and properness differ in general for non-locally compact spaces). (See Figure~\ref{fig: LCSC to acyl}.)

\begin{lemma}\label{Lem:acyl+loc comp implies unif proper_p1} Let $Y$ be an unbounded locally compact metric space. 
\begin{enumerate}
    \item If $\G\to \Isom Y$ is AU-acylindrical, then it is proper. 
    \item If $\G\to \Isom Y$ is uniformly proper, then it is acylindrical. 
    \item Assume $Y$ is uniformly locally compact and that there is some $T>0$ so that for every  $x\in Y$ and $n\in \N$  there is an $y\in Y$ with $d(x,y)\in[nT, (n+1)T]$. If $\G\to \Isom Y$  is acylindrical then it is uniformly proper.
\end{enumerate}
 \end{lemma}

\begin{proof}  The proof of the second statement follows directly from the definitions. The proof of first statement is similar to that of the third, which we prove below. Since properness implies AU-acylindricity by definition, this establishes that AU-acylindricty and properness coincide for $Y$. 

To prove the third statement, let $\e>0$ and let $N(2\e), R(2\e)$ be the associated acylindricity constants for $2\e$. Fix $x\in Y$. Let $n_{R(2\e)} = \min\{n\in \N: nT\geq R(2\e)\}$. By the given condition, there is a $y\in Y$ with $(n_{R(2\e)}+1)T\geq d(x,y)\geq n_{R(2\e)} T\geq R$. 

Set $\rho=\e +(n_{R(2\e)}+1)T$. Since $Y$ is uniformly locally compact, there exist  $C(\rho, \e)\in \N$ and $z_1, \dots, z_m \in \~\hood_\rho(x)$ with  $m\leq C(\rho,\e )$,  forming a finite open cover $\Cup{i = 1}{m}\hood_\e(z_i)\supset \~\hood_\rho(x)$. 

 Let $g, h\in  \set_\e\{x\}$ we note that $g\in h\cdot\cs{2\e} \{x\}$ i.e. $d(h^{-1}gx,x) \leq 2\e$.
Furthermore, $gy\in \~\hood_{\rho}(x)$ since

$$d(x, gy) \leq d(x, gx) + d(gx, gy) \leq \e + d(x,y) \leq  \rho.$$
This means that $gy\in \hood_\e(z_i)$ for some $i\in \{1, \dots, m\}$. If also  $hy\in \hood_\e(z_i)$ for the same $i$, then $d(h^{-1}gy,y) \leq 2\e$ or equivalently $g\in h\cdot\cs{2\e}\{y\}$. 

Let $I\subset \{1, \dots, m\}$ be the set of indices defined by: $j\in I$ if and only if $\hood_\e(z_j) \cap \left(\set_\e\{x\}\cdot y\right)\neq \varnothing$. We have established that $I\neq \varnothing$ since the index $i \in I$. For each $j\in I$ fix $g_j$ with $g_jy \in \hood_\e(z_j)$. Then by the above, we have 
that $\cs\e \{x\}\subset \Cup{j\in I}{}\, g_j\cdot \cs{2\e} \{x,y\}$ and therefore has cardinality at most $C(\rho, \e)\cdot N(2\e)$.
\end{proof}

We shall use the following to deduce the Tits Alternative in later sections. 

\begin{lemma}\label{lem: virt isom to virt for abelian}
Let $1\to K\to \G\to A\to 1$ be a short exact sequence with $K$ finite and $A$ contains an isomorphic copy of $\Z^k$ as a finite index subgroup, for some $k\geq 1$. Then there is a subgroup of finite index $\G'\leq\G$ so that $\G'$ is isomorphic to $\Z^k$.
\end{lemma}

\begin{proof}

Consider the action of $\G$ by conjugation on $K$. This gives a homomorphism $\G\to \Aut K$, the latter is a finite group, and therefore the kernel is of finite index in $\G$. Furthermore, since $A$ contains a finite index subgroup isomorphic to  $\Z^k$, up to passing to a finite index subgroup of $\G$, we may assume that the action of $\G$ on $K$ by conjugation is trivial, and that $A$ is isomorphic to $\Z^k$. (We note that this reduction also amounts to passing to the center of $K$, which is characteristic and finite index in $K$.) 
    
    Let $A$ have  free generators $\{g_1, \dots, g_k\}$. Choose a lift  $\g_i\in \G$ of  $g_i$ for $i=1, \dots, k$. Note that if $k=1$ then $g_1$ has infinite order and hence the sequence is split. 

    Let $p$ be the cardinality of $K$. Then $h^p= 1$ for every $h\in K$. Fix $i, j\in \{1, \dots, k\}$. Then there exists an $h\in K$ such that $\g_i \g_j\g_i^{-1} = \g_j h$.  Since $\g_j$ commutes with $h$ we have that $$\g_i \g_j^p\g_i^{-1} = \g_j^p h^p=\g_j^p.$$
    Therefore, $[\g_i, \g_j^p] = 1$.

Applying the above argument for all pairs, we obtain that $\{\g_1^p, \dots, \g_k^p\}$ are commuting infinite order elements. Their projection to $A$ is a finite index subgroup and hence they generate a finite index subgroup $\G'$ in $\G$.
\end{proof}

\begin{lemma}\label{lem: AU eucl is Zk}
   Consider the action of $\R^\D$ on itself by left-translation. If $\G\to \R^\D$ is AU-acylindrical then it is uniformly proper and $\G$ is virtually  $\Z^k$, for $0\leq k\leq \D$. 
\end{lemma}

\begin{proof}
    First, since the action is AU-acylindrical, the kernel is finite by Lemma~\ref{Lem:AU-acyl has finite kernel}. By Lemma~\ref{Lem:acyl+loc comp implies unif proper_p1} the action is proper and hence the orbits are discrete. By \cite[Lemma 4, p102]{BorevichShafarevich} $\G$ mod the kernel of the action is  $\Z^k$, where $0\leq k\leq \D$. By Lemma~\ref{lem: virt isom to virt for abelian}, $\G$ is virtually $\Z^k$.
\end{proof}

\subsection{Hyperbolic Geodesic Metric Spaces}

The Gromov product of $x,y\in Y$ with respect to the base-point $o\in Y$ is defined as $$\<x,y\>_o=\frac{1}{2}(d(x,o)+d(o,y)-d(x,y))\geq 0.$$ It is a measure of how far the triple is from achieving the triangle equality.

Assume $Y$ is geodesic and consider $a,b,c\in Y$ with choice of geodesics, as in Figure~\ref{Fig: Triangle}. The points $s$ and $t$ are marked on the geodesic $[c,a]$ so that $d(c,b) = d(c,t)$ and $d(a,b)= d(a,s)$. The point $m$ is the midpoint between $s$ and $t$. The points $m'$ and $m''$ are then the corresponding points on the geodesics $[a,b]$ and $[b,c]$ respectively so that 

\begin{eqnarray*}
d(a,m)=d(a,m')\\
d(b,m')= d(b,m'')\\
d(c,m'')= d(c,m)
\end{eqnarray*}

This shows that $d(s,m) = d(m,t) = \<c,a\>_b$ and with the previous observations, we see that
$$d(b,m') = d(b,m'') = \<c,a\>_b.$$

Combining these, we obtain the comparison tripod on the right in Figure~\ref{Fig: Triangle}, which is uniquely determined by  $d(A,M) = d(a,m)$, $d(B,M) = d(b,m')$ and $d(C,M) = d(c,m)$. We note that this yields a natural projection

$$\pi: [a,b]\cup[b,c]\cup[c,a]\to[A,B]\cup[B,C]\cup[C,A]$$

\begin{center}

\end{center}

\begin{figure}
 \includegraphics[width=3.2in]{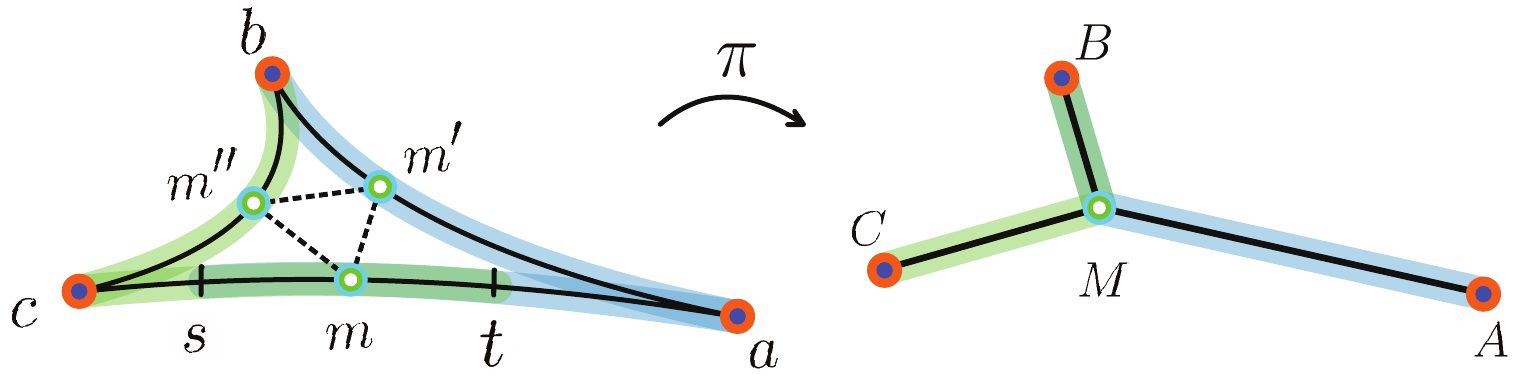}
 \caption{Gromov Product with Comparison Tripod for general triangle}
\label{Fig: Triangle}
\end{figure}

\begin{defn}
A geodesic metric space $X$ is said to be \emph{$\delta$-hyperbolic} if for every $a,b,c\in X$ with comparison tripod as above, the diameter of preimages of a point under $\pi$ is uniformly bounded by $\delta$. 
\end{defn}

\begin{remark}\label{Rem:center of triangle}
For a geodesic triangle as in Figure 1, we shall refer to its \emph{center} as the triple $\{m, m', m''\}= \pi^{-1}\{M\}$. We note that if $X$ is $\delta$-hyperbolic then the center of any geodesic triangle has diameter no larger than $\delta$. There are other (equivalent) formulations of $\delta$-hyperbolicity -- such as the so-called ``slim-triangles" condition due to Rips -- additional details can be found in \cite[Chapter III.H]{BridsonHaefliger}. 
\end{remark}

\begin{cor}\label{Cor: Gromov Prod vs projection}
     Let $X$ be a $\delta$-hyperbolic geodesic metric space. Fix $a,b,c\in X$ and $[a,c]$ a geodesic between $a$ and $c$. If $p\in [a,c]$ is a point closest to $b$ then $\<a,c\>_b\leq d(p,b) \leq \<a,c\>_b + \delta$. 
 \end{cor}

\begin{defn}
    Let $C\geq 1$ and $\lambda\geq 0$. A map $\f: Y\to Z$ between metric spaces is a $(C, \lambda)$ \emph{quasi-isometric embedding} if it is coarsely bi-Lipschitz i.e. for all $y,y' \in Y$ we have
    $$\frac{1}{C}d(y,y')-\lambda\leq d(\f(y),\f(y'))\leq Cd(y,y') + \lambda.$$
    If in addition, $\f$ is \emph{coarsely surjective}, i.e. for every $z\in Z$ there is a $x\in Y$ such that $d(\f(x),z)\leq \lambda$ we  say $\f$ is a quasi-isometry. If $Y= \Z$ then the image of $\f$ is called a \emph{quasigeodesic}.
\end{defn}

Hyperbolicity is a quasi-isometry invariant \cite[Theorem III.H.1.9]{BridsonHaefliger}.

\subsubsection{Gromov Bordification}

Let $X$ be a $\delta$-hyperbolic geodesic metric space. A sequence $x(n)\in X$, $n\in \N$ is said to be a Gromov sequence if $\Lim{n,m\to\8}\<x(n),x(m)\>_o =\8$.  Recall that if $a(n,m)\in \R$ is a double indexed sequence then $\Lim{n,m\to\8}\, a(n,m) = \8$ if for every $L>0$ there is an $N>0$ so that if $n,m\geq N$ then $a(n,m)\geq L$. Two Gromov sequences $x(n), y(n), n\in \N$ are said to be equivalent if $\Lim{n,m\to\8}\<x(n),y(m)\>_o =\8$. 

The Gromov boundary as a set is the collection of equivalence classes of Gromov sequences converging to infinity and is denoted $\partial X$. If $\xi, \eta\in \partial X$ then we may extend the Gromov product by taking the following infimum ranging over equivalence class representatives:

$$\<\xi,\eta\>_o= \inf \left(\Lim{n,m\to \8}\<x(n),y(m)\>_o\right).$$

We topologize $\partial X$ by saying that two classes are topologically ``nearby" if they have ``large" Gromov product, or equivalently the corresponding quasigeodesic rays  fellow travel for a ``long'' time.

Since these conditions are independent of the choice of the base point $o\in X$  there is a natural map  \cite[Theorem 5.3]{Vaisala} (see also \cite{Hamann})
$$\Isom X\to \Homeo(\partial X), f\mapsto \partial f.$$

\begin{remark}\label{Rem: (1, 20delta)_p1} Benakli and Kapovich sketch an argument  that for every $x\in X$ and $\xi\in \partial X$, there is a $(1, 10\delta)$ quasigeodsic ray from $x$ to $\xi$. They also state that this can be used to prove that a pair of distinct boundary points  can be connected by a $(1, 20\delta)$ quasigeodesic \cite[Remark 2.16]{BenakliKapovich}. This sketch can be more rigorously seen by results from V\"ais\"al\"a's work \cite{Vaisala} as follows: By Remark 6.4, a $(\mu,h)$-road can be made into a $(1, \mu)$-quasigeodesic ray. By Theorem 6.7, for every $h>0$ there exists a $(4\delta+2h,h)$-road connecting $x$ and $\xi$. Hence there exists a  $(1, 5\delta)$ quasigeodesic ray connecting $x$ and $\xi$. Similarly, by Lemma 6.13, given distinct $\xi, \eta \in \partial X$, for every $h>0$ there exists a $(12\delta +10 h, h)$-biroad and so $\xi$ and $\eta$ can be connected by a $(1, 13\delta)$ quasigeodesic. 
The work of Bonk and Schramm  can also be employed here with less precision to the constants \cite[Proposition 5.2]{BonkSchramm}. The fact that the multiplicative constant here is trivial is indispensable for our results and therefore most will be phrased for $(1, \lambda)$ quasigeodesics. We note that there is a small discrepancy in that those sources use $\R$ to define quasigeodesics whereas we use $\Z$. This means that technically, the additive constants above should be corrected with a $+1$, where appropriate. For ease of notation, we shall ignore this discrepancy as well as the precision guaranteed by V\"ais\"al\"a: we will utilize $(1, 10\delta)$ quasigeodesic rays and $(1, 20\delta)$ quasigeodesics. 
\end{remark}

We shall denote by $\~ X = X\cup\partial X$ the Gromov bordification of $X$ and  note that it may not be compact, and $\partial X$ not closed, though they are both metrizable \cite[Section 5]{Vaisala}. It is for this reason that we use the term bordification. 

\subsubsection{The Morse Property}

An important property of hyperbolic metric spaces is the Morse property for stability of quasigeodesics. 

\begin{prop}[Finite Morse Property]\cite[Theorem 3.7]{Vaisala}\cite[Theorem III.H.1.7]{BridsonHaefliger}\label{prop:qg triangles slim hyp}  Let $X$ be a $\delta$-hyperbolic space, and $C\geq 1, \lambda\geq 0$ be two constants. Then there is a constant $M = M(\delta, C, \lambda)$ such that if $q$ is a $(C, \lambda)$-quasigeodesic in $X$ and $[x,y]$ is a geodesic segment between the end points of $q$, then the Hausdorff distance between $q$ and $[x,y]$ is at most $M$.
\end{prop}

\begin{theorem}[Infinite Morse Property]\cite[Theorem 6.32]{Vaisala}\label{Slim Q-Bigons infinite_p1}
Let $X$ be $\delta$-hyperbolic, $C\geq 0$ and $\lambda \geq 0$. There exists $M= M(\delta,C,\lambda)$ so that if   $c, q$ are either two $(C, \lambda)$ quasigeodesic rays  starting from the same point, and converging to the same point in $\partial X$ or are two bi-infinite $(C, \lambda)$ quasigeodesics with the same endpoints on $\partial X$, then $d_{Haus}(im(c), im(q))\leq M$. 
\end{theorem}

\begin{remark}[The Morse Constant]\label{rem: Morse constant}
    In V\"ais\"al\"a's work, the constant  $M(\delta,C,\lambda)$ from Proposition~\ref{prop:qg triangles slim hyp} is used to prove Theorem~\ref{Slim Q-Bigons infinite_p1} and are therefore the same. We shall exclusively be using $(1, \lambda)$ quasigeodesics and later want to refer to ``the" Morse constant. So, to be precise we set $M_\lambda$ to be $\delta$ plus the infimum of the $M(\delta, 1, \lambda)$ from Proposition~\ref{prop:qg triangles slim hyp}.
\end{remark}

A consequence of the Morse Property is that the ``slim-triangles" formulation of hyperbolicity may be extended to quasigeodesic triangles as well; see \cite[Corollary III.H.1.8]{BridsonHaefliger} for instance.

The following lemma establishes a fellow-traveling property for $(1, \lambda)$ quasi-geodesics.  We note that the lemma is not true for $(C,\lambda)$ quasigeodesics if $C \neq 1$. For example, consider the quasigeodeisc $c, q: \Z\to \R$, where $c(n)= n$ and $q(n) = C n$. Then $d(c(n),q(n)) = |(C-1) n|$ which is unbounded if $C \neq 1$.

\begin{lemma}\label{lemma: fellow travel q-rays}
  Let $c, q: \N\to X$ be $(1,\lambda)$ quasigeodesic rays with $c(0)=q(0)$ and $c(n), q(n) \to \xi \in \partial X$. Let $M_\lambda$ be as in Remark~\ref{rem: Morse constant}. Then $d(c(n),q(n))\leq 2M_\lambda + 3\lambda$ for all $n\in \N$. 
\end{lemma}

\begin{proof}
 Fix $n\in \N$. By Theorem~\ref{Slim Q-Bigons infinite_p1} there is an $m_n\in \N$ so that $d(c(n),q(m_n))\leq M_\lambda$. Since $c$ and $q$ are quasigeodesics with the same base point, using the reverse triangle inequality we have that
 $$|n-m_n| -2\lambda \leq |d(c(n),c(0)) -d(q(m_n),q(0))| \leq d(c(n),q(m_n))\leq M_\lambda.$$
 Therefore
\begin{eqnarray*}
d(c(n),q(n))&\leq& d(c(n),q(m_n))+d(q(m_n),q(n))\\
&\leq& M_\lambda + |m_n-n| +\lambda \\
&\leq& 2M_\lambda + 3\lambda.
\end{eqnarray*}
\end{proof}

\begin{cor}\label{cor: (1,mu)-fellow travel}
 Let $c, q: \Z\to X$ be $(1,\lambda)$ oriented quasigeodesics  converging to the same points in $\partial X$, in the same direction.  Let $M_\lambda$ be the Morse constant for $(1, \lambda)$ quasigeodesic rays and $M'_\lambda$ be the Morse constant for $(1, \lambda +2M_\lambda)$ quasigeodesics. Then  for every $n\in \Z$ we have that $$d(c(n), q(n))\leq  2M'_\lambda + 8 M_\lambda +5\lambda+ d(c(0), q(0)).$$
\end{cor}

\begin{proof}
By Theorem~\ref{Slim Q-Bigons infinite_p1}, there exists $m_0\in \Z$ such that $d(c(0), q(m_0))\leq M_\lambda$. 
Let $q':\Z\to X$ be given by $q'(n)= q(n+m_0)$, which is still a $(1,\lambda)$ quasigeodesic. Furthermore, define $c'(n) = c(n)$ if $n\neq 0$ and $c'(0)=q'(0)$. Then the quasigeodesics $c'|_{[0,\8)}$ and $q'|_{[0,\8)}$ (respectively $c'|_{(-\8,0]}$ and $q'|_{(-\8,0]}$) are  $(1, \lambda +M_\lambda)$ quasigeodesic rays starting from the same point and converging to the same end point in $\partial X$. By Lemma~\ref{lemma: fellow travel q-rays} for each $n\in \N$ we have that 
$$d(c'(n), q'(n))\leq 2M'_\lambda + 3(\lambda + 2M_\lambda)$$
\begin{center}
    and
\end{center} 
$$d(c'(-n), q'(-n))\leq 2M'_\lambda + 3(\lambda + 2M_\lambda).$$

Now, let $n\in \Z$. Then
\begin{eqnarray*}
    d(c(n), q(n)) &\leq& d(c(n), c'(n))+ d(c'(n), q'(n)) +d(q'(n),q(n))\\
    &\leq & M_\lambda+ 2M'_\lambda + 3(\lambda + 2M_\lambda) + |m_0|+\lambda
    \\
    &\leq & 2M'_\lambda +7M_\lambda +4\lambda + d(q(m_0), q(0)) +\lambda
     \\
    &\leq &2M'_\lambda +7M_\lambda +5\lambda + d(q(m_0), c(0)) + d(c(0), q(0)) 
    \\
    &\leq &   2M'_\lambda +8M_\lambda +5\lambda +d(c(0), q(0)) 
     \\
\end{eqnarray*}

\end{proof}
 
 Observe that the proofs of Lemma~\ref{lemma: fellow travel q-rays} and Corollary~\ref{cor: (1,mu)-fellow travel} also hold for $(1, \lambda)$ quasigeodesic segments by applying the finite version of the Morse Property Proposition~\ref{prop:qg triangles slim hyp}. Furthermore, if we know the quasigeodesics in question are at bounded Hausdorff distance, then hyperbolicity condition on $X$ is unnecessary.

\subsection{Products}\label{Subsec:Products}
We begin by recalling the universal property for products and some consequences and then dive in to the context of this work: finite products of $\delta$-hyperbolic spaces.

The universal property for direct products is important as it encapsulates that a variety of objects associated to a product (e.g. a point in the product, or maps \emph{to} a product) are determined by its projections to the factors.

\subsubsection{The universal property}\label{subsec: Univ Prop}

 Recall that if $A_1,A_2,B$ are sets, and $p_i:A_1\times A_2\to A_i$,  is the natural projection defined by, $p_i(a_1,a_2):=a_i$  for $i=1, 2$ then for any set-maps $f_1:B\to A_1$ and  $f_2:B\to A_2$ there exists a unique set-map $f: B\to A_1\times A_2$ such that $p_1\circ f= f_1$ and $p_2\circ f= f_2$. This fits into the following commutative diagram. 
\begin{center}
\begin{tikzcd}
	{A_2} & {A_1\times A_2} \\
	B & {A_1}
	\arrow["{{{{p_2}}}}"', from=1-2, to=1-1]
	\arrow["{{{{p_1}}}}", from=1-2, to=2-2]
	\arrow["{{{{f_2}}}}", from=2-1, to=1-1]
	\arrow["{{{\exists ! f}}}"{description}, dashed, from=2-1, to=1-2]
	\arrow["{{{{f_1}}}}"', from=2-1, to=2-2]
\end{tikzcd}

\end{center}

\noindent
\textbf{\underline{Notation}:} Given a map $f:B\to A_1\times A_2$, we shall denote the factors as $f_1$, and $f_2$, respectively. Similarly, given $f_i:B\to A_i$, for $i=1, 2$, we may assemble this into $f: B\to A_1\times A_2$ in a unique fashion. However, if we wish to emphasize this assembly, we may write $\Delta(f_1, f_2): B\to A_1\times A_2$, where the symbol $\Delta$ represents the word ``diagonal". 

Note that if $B=B_1\times B_2$ is itself a product, there is no reason that $f_1$ must factor through (one or both of) the natural projections $p_i': B_1\times B_2\to B_i$ for $i=1, 2$. For example, consider a rotation $r:\R\times \R\to \R\times \R$, given by $r_1(x_1,x_2) = x_1\cos \theta -x_2\sin\theta$, and $r_2(x_1,x_2) = x_1\cos \theta +x_2\sin\theta$. The universal property states that $r(x_1,x_2)$ is uniquely determined by  $r_1$ and $r_2$. Note that both maps $r_1$ and $r_2$ depend on both variables. \emph{Such a map \underline{does not} preserve the factors.}

However, given set-maps $f_i: B_i\to A_i$, for $i=1, 2$ we do have a unique set map $f_1\times f_2: B_1\times B_2\to A_1\times A_2$ given by $(f_1\times f_2)(b_1, b_2) = (f_1(b_1),f_2(b_2))$, which \emph{\underline{does} preserve the factors}. The properties of $f_1\times f_2$ are captured by the following commutative diagram. Note the absence of maps $B_1\to A_2$ and $B_2\to A_1$.

\begin{center}
   \begin{tikzcd}
	{A_2} && {A_1\times A_2} \\
	{B_2} & {B_1\times B_2} \\
	& {B_1} & {A_1}
	\arrow["{p_2}"', from=1-3, to=1-1]
	\arrow["{p_1}", from=1-3, to=3-3]
	\arrow["{f_2}", from=2-1, to=1-1]
	\arrow["{\exists ! f_1\times f_2}"{description}, dashed, from=2-2, to=1-3]
	\arrow["{p_2'}"', from=2-2, to=2-1]
	\arrow["{p_1'}", from=2-2, to=3-2]
	\arrow["{f_1}", from=3-2, to=3-3]
\end{tikzcd}
\end{center}

Furthermore, the above commutative diagrams can be completed uniquely in a variety of categories, yielding unique homomorphisms for groups, continuous maps for topological spaces (with the product topology), and isometries (with respect to, for example, $\ell^p$-product metrics). 
For clarity, we note that in the image above, $$f_1\times f_2= \Delta(f_1\circ p_1',f_2\circ p_2').$$ 

\noindent
\textbf{\underline{Notation}:} When considering multiple factors, for example $\D>1$, and $f_i: B_i\to A_i$, for $i=1, \dots, \D$, we shall denote the above factor preserving product map as $$\Prod{i=1}{\D}f_i: \Prod{i=1}{\D}B_i \to\Prod{i=1}{\D}A_i$$

\subsubsection{Products of $\delta$-hyperbolic spaces and their isometry groups}
Finite products of $\delta$-hyperbolic spaces of course play a leading role in our work. We shall consistently use the notation that $\X=\Prod{i=1}{\D} X_i$, where $X_i$ are separable complete geodesic $\delta$-hyperbolic spaces. Separability and completeness will ensure that  their isometry groups are Polish (see Section~\ref{Sec:toponisomgroup}). In \cite{BFPart2}, these conditions will give us additional control of the spaces, allowing us to build an ``essential core" (see Proposition~\ref{prop: ess core}).

We denote by $\mathrm{Sym}_\X(\D)$ the collection of possible permutations of indices according to whether the corresponding spaces are isometric and fix such an isomorphism. This allows us to consider the corresponding automorphism group $\Aut\X :=\Sym_\X(\D)  \ltimes \Prod{i=1}{\D} \Isom X_i$ and we will denote by $\Aut_{\!0}\X= \Prod{i = 1}{\D} \Isom X_i$, which is the maximal subgroup which preserves the factors. 

\noindent
\textbf{\underline{Notation}:} Let $i \in \{1, \dots, \D\}$. For $x\in \X$, we shall denote the $i^\text{th}$ coordinate of $x$ as $x_i\in X_i$. For  $g\in \Aut_{\! 0}\X$, we shall denote the isometry that $g$ acts by on $X_i$ by $g_i\in \Isom X_i$. This does mean that we shall occasionally have to use notation such as $x^n$ or $x(n)$ to denote finite or infinite sequences of elements. 

There are of course different metrics that one could consider on $\X$, the class of $\ell^p$-product metrics being an important family of equivalent ones. The availability of the $\ell^2$-product metric allows us to place $\X$ in the category of non-positively curved spaces. However, we  will most often use the $\ell^\infty$ metric, and while we distinguish these in the following definition, we shall use $d$ for the $\ell^\infty$ metric, as it will be the one most utilized. We shall more generally, use $d$ to denote most metrics, unless confusion may arise. 

\begin{defn}[The $\ell^p$-product metric on $\X$]\label{def: lp metric}
 If  $x,y \in \X$ and $p\in [1,\8)$ then $d_p(x,y):=\sqrt[^p]{\Sum{i=1}{\D}[d_i(x_i,y_i)]^{^p}}$ and if $p=\8$ then $d_\8(x,y):=\max{i\in \{1,\dots, \D\}} \hspace{2pt} d_i(x_i,y_i).$
\end{defn}

We note that if the factors are geodesic then $\X$ is also geodesic with respect to the $\ell^p$-product metric (e.g. consider the product of geodesics in each factor and choose an $\ell^p$-geodesic in the resulting $\D$-dimensional interval). This is true in particular for $p=2$, allowing us to think of these spaces as being ``nonpositively curved". In Part II \cite{BFPart2}, we shall prove that under natural hypotheses, the $\ell^2$-isometry group of $\X$ is $\Aut\X$.

Using the Roller compactification and boundary of a CAT(0) cube complex as inspiration, we shall consider $\~{\X}:= \Prod{i=1}{\D} \~X_i$ and $\partial \X := \~{\X}\setminus \X$, which is only a bordification in general. (Recall that $\~ X= X\cup \partial X$ represents the Gromov bordification of the hyperbolic space $X$.) Note that  if $\D=2$ then $\partial \X= \big( X_1\times \partial X_2\big ) \cup\big(\partial X_1\times \partial X_2\big ) \cup  \big(\partial X_1\times X_2\big )$.
We define the \emph{regular boundary} as the subset $\partial_{reg}\X = \Prod{i=1}{\D} \partial X_i$. As is the case for the Roller compactification of a CAT(0) cube complex (e.g. \cite{NevoSageev,CFI, KarSageev,FernosFPCCC, FLM, FLM2}, also \cite{Link}) the regular boundary carries  the important dynamical information. We note that usually the regular boundary is referred to as the set of ``regular points". However, as none of these boundaries confer a compactification of $\X$ in general, we believe the use of the word ``boundary" here is acceptable. 

Consider the projection $p:\Aut \X\to \Sym_\X(\D)$. Then  $\ker(p)=\Aut_{\! 0}\X= \Prod{i=1}{\D}\Isom X_i$ is a finite index subgroup. We then have that if $\rho:\G\to \Aut\X$ is a homomorphism then  $\G_0 = \rho^{-1}(\ker(p))$ is of finite index in $\G$. 

As a product, $\Aut_{\! 0}\X$ comes equipped with natural projections $\pi_i:\Aut_{\! 0}\X\twoheadrightarrow \Isom X_i$ for each $i=1,\dots,\D$. 
As was discussed in Section~\ref{subsec: Univ Prop}, $\rho|_{\G_0}$ is uniquely determined by the factor actions  $\rho_i=\pi_i\circ \rho:\G_0\to \Isom X_i$, for $i=1,\dots,\D$, and conversely, given a homomorphism $\rho_i: \G_0\to \Isom X_i$ for each $i=1, \dots, \D$ there exists a unique homomorphism $\rho: \G_0\to \Aut_{\! 0}\X$ such that $\rho_i=\pi_i\circ \rho$, namely $\Delta(\rho_1, \dots, \rho_\D)$. The related question of inducing representations from a finite index subgroup will be considered in Section~\ref{Sec: rep theory basics}.

\section{Isometric Actions in Rank-one}\label{sec:isomrank1}

In order to classify the isometries in higher rank in the next section, we first discuss isometries and actions in rank-one here.

\subsection{Classification results.}

We recall some standard terminology used in relation to actions on hyperbolic spaces and introduce some new terminology along the way. Note that the following classification extends that obtained by Gromov in the case when the boundary is ``visible by geodesics", which includes the locally compact case \cite[Corollary 8.1.B]{Gro} and by Hamann in the general case \cite[Theorem 2.3, Proposition 4.4]{Hamann}. Recall that given a hyperbolic space $X$, we denote by $\partial X$ its Gromov boundary.

\begin{defn}\label{defn:elementsclassextn}
 Let $X$ be a (unbounded) $\delta$-hyperbolic space, and $\g\in \Isom X$. For $\e\geq 0$, define $\O^\e(\g)= \{x\in X: d(\g^n x, x) \leq \e \hspace{2pt} \forall n\in \Z\}$. We shall say that $\g$ is 

 \begin{enumerate}
 \item \emph{elliptic} if all orbits are bounded. More precisely, an elliptic isometry is 
\begin{enumerate}
    \item a \emph{tremble} if there is an $\e\geq 0$ such that the set $\O^\e(\g)=X$;
  \item a \emph{rotation} if $\O^\e(\g)$ is bounded for every $\e\geq 0$; 
   \item a \emph{rift} if it is neither a tremble nor a rotation. 
  \end{enumerate}

  \item \emph{parabolic} if $\langle \g \rangle$ has a unique fixed point in $\partial X$ and has unbounded orbits.

  \item \emph{loxodromic} if $\g$ is not elliptic and has exactly two fixed points on $\g^-, \g^+\in \partial X$, in which case $\Lim{n\to \pm \8}\g^n.x= \g^\pm$, for all $x\in \~ X\setminus \{\g^-, \g^+\}$.
 \end{enumerate}
\end{defn}

\begin{remark} For $\g$ loxodromic, the fixed point corresponding to $\g^-$ and $\g^+$  will be called the \emph{repelling} and \emph{attracting} fixed points of $\g$. Hamann guarantees the convergence for $x\in X$ in Theorem 2.3, and the extension to the boundary is Property (C1) \cite{Hamann}.
Furthermore, loxodromic elements are characterized by having positive \emph{stable translation length}, i.e. $\Inf{n\in \N}\,\dfrac{d(\g^n x,x)}{n}>0$ for all $x\in X$ and quasi-isometrically embedded orbits. 
\end{remark}

\begin{wrapfigure}[10]{r}{.28\textwidth}\vskip-.15in
\includegraphics[width=0.28\textwidth]{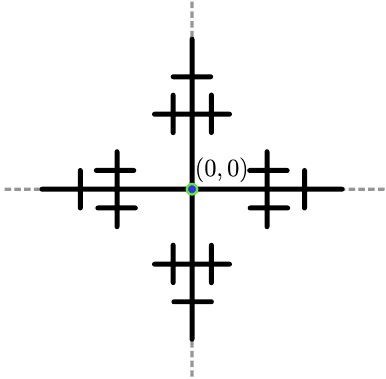}
    \label{fig:placeholder}
\end{wrapfigure}

\noindent
\textbf{Example:} To distinguish between the three types of elliptic elements, consider the standard embedding of the 4-valent tree $T$ in $\R^2$, as depicted in the adjacent figure. A rotation by $\pi/2$ would be a rotation on the tree,  and a flip along a horizontal axis would be a rift. On the other hand, trees without leaves have no non-trivial trembles. 

To create a non-trivial tremble, we may take the product of $T$ with a bounded graph. Then any nontrivial isometry of the bounded graph will induce a non-trivial tremble of the $\delta$-hyperbolic space $T\times C$. \hfill $\blacksquare$

\medskip
 Given an action $\rho:\G \to \Isom X$,  we denote by $\Lambda (\rho(\G))$ the set of limit points of $\rho(\G)$ in $\partial X$. That is, $$\Lambda (\rho(\G))=\partial X\cap \overline{\rho(\G) .x},$$ where $\overline{\rho(\G) .x}$ denotes the closure of the $\rho(\G)$-orbit of $x\in X$ in $\~X$; the limit set is independent of the choice of basepoint $x$. The $\rho$-action of $\G$ on $X$ naturally extends to a continuous action of $\G$ on $\partial X$ by postcomposing $\rho$ with $\partial: \Isom X\to \Homeo(\partial X)$.
 We also note that as $\partial X$ is not closed in general, neither is $\Lambda(\rho(\G))$. Moreover, we will often drop the name of the homomorphism $\rho$ and simply denote the set by $\Lambda(\G)$, if no confusion can arise.

The following theorem summarizes the standard classification of groups acting on hyperbolic spaces. Since the image of a homomorphism is a subgroup, we may also view this as a classification of subgroups of $\Isom X$. 

\begin{theorem}\label{thm:actionsclassrk1} Let $X$ be a hyperbolic space and $\rho:\G\to \Isom X$ a homomorphism. For $\e\geq 0$ w define $\O^\e(\rho(\G))=\{x\in X: d(\rho(g)x,x)\leq \e\text{ for all } g\in \G\}$. One of the following is true:
 \begin{enumerate}
 \item $|\L(\rho(\G)) |= 0$ and the action is \emph{elliptic}, i.e. $X=\Cup{\e\geq 0}{}\O^\e(\rho(\G))$. If in addition
 \begin{itemize}
     \item there exists $\e\geq 0$ such that $\O^\e(\rho(\G))=X$ then we say the action is a \emph{tremble};
     \item for every $\e\geq 0$ we have  that $\O^\e(\rho(\G))$ is bounded then we say the action is a \emph{rotation};
     \item for all $\e\geq 0$ sufficiently large we have that $\O^\e(\rho(\G))\neq X$ and unbounded, then we say the action is a \emph{rift}.

 \end{itemize}

  \item $|\L(\rho(\G))| = 1$ and the action is \emph{parabolic}. In this case,  $\rho(\G)$ has a unique fixed point in $\partial X$, unbounded orbits, and   no loxodromic elements;
   \item $|\L(\rho(\G))| = 2$ and the action is \emph{lineal}. In this case,  a loxodromic exists in $\rho(\G)$.  If further, $\rho(\G)$ fixes each limit point, the action is called \emph{oriented lineal}.
  
  \item $|\L(\rho(\G))| = \8$  and  the action is one of the following:
  \begin{enumerate}
    \item \emph{quasiparabolic} if in addition $\rho(\G)$ fixes a unique point in $\partial X$. In this case $\rho(\G)$ contains two loxodromic with distinct attracting fixed points and that generate a free semigroup;
      \item \emph{general type} if there exists two loxodromics in $\rho(\G)$   with disjoint fixed point sets  in $\partial X$ that freely generate a free subgroup of $\G$ acting by loxodromic elements.
  \end{enumerate}
 \end{enumerate}
\end{theorem}

\begin{remark}
    This theorem was proved by Gromov in the locally compact case, or when the boundary is accessible by geodesics, where it is argued that any $\delta$-hyperbolic space can be ultra-completed in an $\Isom X$-equivaraint fashion, i.e. completed so that it is geodesic and has a  boundary that is accessible by geodesics \cite[Section 8.2]{Gro}. Bonk and Schramm proved that a hyperbolic space can be completed through the use of transfinite induction \cite[Theorem 4.1]{BonkSchramm}, a procedure unlikely to be compatible with the isometry group. The general case for a geodesic $\delta$-hyperbolic space (not necessarily separable) follows from a combination of \cite[Theorem 2.8, Theorem 4.4]{Hamann} and \cite[Proposition 7.3.1, Proposition 10.5.4]{DasSimmonsUrbanski}. We contribute to the classification through the observation that elliptic actions in general naturally have the above trifurcation. 
\end{remark}

The reader may have noted the figure on the title page, which has a representation for each of the types of actions in the following theorem. This schematic should be read as one would read figures in the disk model of the hyperbolic plane. We hope that, beyond its aesthetic appeal, it is also instructive (see \cite{CoulonDorfsman-HopkinsHarrissSkrodzki} ``On the importance of illustration for mathematical research").

\subsection{Busemann quasimorphism}

A function $b\colon \G \to \mathbb R$ is a \emph{quasimorphism} if there exists a constant $C\geq 0$   such that $$|b(gh)-b(g)-b(h)|\le C$$ for all $g,h\in \G$. We say that $b$ has \emph{defect at most $C$}. If, in addition, the restriction of $b$ to every cyclic subgroup of $G$ is a homomorphism, then $b$ is called a homogeneous quasimorphism. Every quasimorphism $b$ gives rise to a homogeneous quasimorphism $\beta$ defined  as follows:
$$
\beta(g)=\lim_{n\to \infty} \frac{b(g{^{n}})}n \hspace{5pt} \forall g \in \G.
$$

We note that since $b(g^n)$ is uniformly close to a subadditive sequence this limit exists by Fekete's Lemma. 
The function $\beta$ is called the \emph{homogenization of $b$.} It is straightforward to check that
$$
 |\beta(g) -b(g)|\le C
$$
for all $g\in \G$ where $C$ is the defect of $b$ as above.

Given any action of a group on a hyperbolic space fixing a point $\xi$ on the boundary, one can associate the \emph{Busemann quasimorphism}. We briefly recall the construction and necessary properties here, and refer the reader to \cite[Sec. 7.5.D]{Gro} and \cite[Sec. 4.1]{Man} for further details.

\begin{defn}\label{Bpc}
Let $\G\to \Isom X$ be an action that fixes a point $\xi\in \partial X$.
Let $s=\{x(n)\}_{n\in \N}$ be any sequence of points of $X$ converging to $\xi$. Then  the function $b_{s}\colon \G\to \mathbb R$ defined by
$$
b_{s}(g)=\limsup\limits_{n\to \infty}\left(d (gx(0), x(n))-d(x(0), x(n)\right)
$$
is a quasimorphism. Its homogenization $\beta_{s}$ is called the \emph{Busemann quasimorphism}. It is known that this definition is independent of the choice of $s$ (see \cite[Lemma 4.6]{Man}), and thus we can drop the subscript and just write $\beta$.\end{defn}

Roughly speaking, the Busemann quasimorphism measures the translation of group elements towards or away from $\xi$, ignoring horocyclic translation. In particular $\b(g) \neq 0$ if and only if $g$ is loxodromic. Therefore, $\b \neq 0$ if the action is oriented lineal or quasiparabolic.  Note that, in general, $\b$ need not be a homomorphism, but it is when the group is amenable or if $X$ is proper (see \cite[Corollary 3.9]{Amen}). 

We will use the following results concerning the Busemann quasimorphism, which appear in \cite[Section 3]{ABR} in a slightly different setting. We include the setup, notation and statements here -- these will be reused later, particularly in Section~\ref{sec:elemsubgrp}. We omit some of the proofs as they follow directly from the proofs in \cite{ABR}.

\noindent
\textbf{\underline{Local Notation}:}  Let us establish the notation that will be used through until the end of the proof of Lemma \ref{lemmaelimpara}, and where these results are invoked. Let $X$ be $\delta$-hyperbolic,  $\xi \in \partial X$, $q: \Z\to X$ be a quasigeodesic converging to $\xi$ in the positive direction, and $M= M_{20\delta}$ be the Morse constant as in Remark \ref{rem: Morse constant}. If $g\in\fix\{\xi\}$ then the ray $q|_{[0,\infty)}$ and its translate $gq|_{[0,\infty)}$ are both $(1,20\delta)$ quasigeodesic rays that share the endpoint $\xi$ and thus  are eventually $M$-Hausdorff close by \cite[6.9.Closeness Lemma]{Vaisala}, and synchronously fellow travel forever after (see Lemma~\ref{lemma: Unif Bound}). Specifically, there are numbers $t_0=t_0(g)$ and $s_0=s_0(g)\geq 0$ depending on $g$ and $q(0)$ so that $q|_{[t_0,\infty)}$ and $gq|_{[s_0,\infty)}$ are $M$-Hausdorff close and $d(q(t_0),gq(s_0))\leq M$.   In other words $s_0$ is a bound for how long we must wait for the ray $gq|_{[0,\infty)}$ to become close to the ray $q|_{[0,\infty)}$. This depends only on $d(q(0),gq(0))$, and $s_0(g)$ may be chosen smaller than a (linear) function of $d(q(0),gq(0))$. We consider the difference $l=t_0-s_0$ as the amount that $g$ ``shifts'' the quasigeodesic $q$, which may be positive or negative. In the case when the space $X$ is a quasi-line, we may take $s_0 = 0$ since $q, gq$ share both end points on $\partial X$ and are thus always Hausdorff-close.

\begin{lemma}\cite[Corollary 3.15]{ABR}
\label{Lem:shiftfn}
Under the setup described above, there is a constant $\consttwo>0$ so that for any $g\in \G$, if $s\geq s_0(g)$ then $$d\big(q(s-\b(g)),gq(s)\big)\leq \consttwo.$$ 

In particular, when $X$ is a quasi-line, then $d\big(q(-\b(g)), gq(0) \big)\leq \consttwo$ and it follows that $d(q(0), gq(0)) \leq |\b(g)| + \consttwo + 20 \delta$. 
\end{lemma}

\begin{prop}\cite[Proposition 3.12]{ABR}\label{prop:smalltranslation}
Let $X$ be a quasi-line, $\xi\in \partial X$ and $\G\to \fix\{\xi\}$ be an action so that the associated Busemann quasimorphism $\b \colon \G\to \R$ is a  homomorphism. Define the function $A\colon [0,\8)\to \R$ as $$A(r) = \op{sup}\{ |\b(g)| \: g\in \G \text{ s.t. } d(q(0) ,gq(0)) \leq r\}.$$ Then there exists a constant $B>0$ such that for any $g, h \in \G$ with $d(q(0),gq(0))\leq r$ and $\b(h)\leq A(r)$, we have $d(gh q(0),hq(0))\leq |\b(g)|+B$.
\end{prop}

We will also need the following lemma pertaining to the behavior of quasigeodesics in parabolic actions. If the action is parabolic, then $\G$ is infinite, and we may fix  a strictly increasing sequence $\{1\} \subsetneq S_1 \subsetneq S_2 \subsetneq \dots$ of finite symmetric subsets of $\G$. 

\begin{lemma}\label{lemmaelimpara}  Suppose $\G\to \fix \{\xi\}$ is parabolic and  $q: \N \to X$ is a $(1, 20\delta)$ quasigeodesic ray  converging to $\xi$. Let $E$ be the constant from Lemma~\ref{Lem:shiftfn}. Define $$\mathrm{M}(t) = \op{max} \{ m \mid d( gq(t) , q(t) ) \leq E \hspace{5pt} \forall g \in S_m \}.$$ Then for every $n \in \mathbb{N}$, there is a $t_n$ such that $\mathrm{M}(t) \geq n$ for all $t \geq t_n$. In particular, $\displaystyle \lim_{t \rightarrow \infty} \mathrm{M}(t) = \infty$. 
\end{lemma}

\begin{proof} Since the action is parabolic, there are no loxodromic elements and hence the Busemann quasimorphism $\b \equiv 0$ on $\G$ (and is hence a homomorphism).

Since the action is by isometries fixing $\xi$, if $q:\N\to X$ is a $(1, 20\delta)$ quasigeodesic converging to $\xi$, then so is $gq$ for every $g \in \G$. Fix $n\in \N$ a let $g_1, g_2\in S_n$. Then  there exists a $t_0 = t_0(g_1, g_2)$ such that $d(g_1q(t), g_2q(t)) \leq E$ for all $t \geq t_0$; this follows from Lemma~\ref{Lem:shiftfn}. Set $t_n = \op{max} \{t_0(g_1, g_2)\mid g_1 , g_2 \in S_n\}$. It follows that if $t \geq t_n$ then $S_n$ satisfies the definition of $\mathrm{M}(t)$ and therefore $\mathrm{M}(t) \geq  n $. 
\end{proof}

\section{Actions in Higher rank}\label{sec:isomhigherrank_p1}

We now turn our attention to classifying isometries in higher rank. Note that some of these have been considered in the work of Button \cite[Section 4]{buttonhigherrank}. 

\begin{defn}
An element $g=(g_1, \dots, g_{\D})\in \Aut_{\! 0}\X$ is said to be
\begin{enumerate}
\item \emph{elliptic} if all components $g_i$ are elliptic;
\item \emph{parabolic} if all components $g_i$ are  parabolic;
\item \emph{pseudo-parabolic} if all components $g_i$ are either elliptic or parabolic, with at least one of each type;
\item \emph{regular loxodromic} if all components $g_i$ are loxodromic;
\item \emph{weakly-loxodromic} if at least one component $g_i$ is loxodromic, but $g$ is not regular. 

\end{enumerate}
\end{defn}

\subsection{Obstructions and simplifications to acylindricity in higher rank.}\label{sec: obstructions/simplif}  Our first goal is to show that factors where the action is elliptic or parabolic do not ``contribute" to acylindricity, and can thus be omitted from $\X$ without losing the acylindricity of the action (see Lemmas~\ref{Lem:elimellfactors} and~\ref{prop:removepara}). For the case of AU-acylindricity, we note that Lemma~\ref{Lem:elimellfactors} shows that elliptic factors can still be omitted, but the remaining results require the stronger acylindricity assumption. 

Since the notion of an elliptic action is well defined for any metric space, we prove the following lemma in a more general setting. 

 \begin{lemma} \label{Lem:elimellfactors} Suppose that $\G \to \Isom Y \times \Isom Z$ is (AU-)acylindrical, where $Y,Z$ are metric spaces. If the projection $\G \to \Isom Y$ is elliptic then the projection $\G \to \Isom Z$ is (AU-)acylindrical. 
\end{lemma}

\begin{proof} We shall just prove the case of an acylindrical action since the case of nonuniformly acylindrical actions is analogous.  Let $\e >0$ be given. Fix $y \in Y$. Then there is a constant $B >0$ such that $\op{diam}(\G y) \leq B$. Let $R, \mfN$ be the acylindricity constants for $\e' = \e +B$ for the action on $Y \times Z$. Let $a,b \in Z$ be any points such that $d(a,b) > R$. Consider the set $$S = \{g \in \G \mid d(a, ga) \leq \e \text{ and } d(b, gb) \leq \e \}.$$ We claim that $|S| \leq \mfN$, which will prove acylindricity.  

Firstly, observe that \begin{align*}
    d( (y ,a), (y, b)) &= d(y ,y) + d(a, b) \\
    &= d(a, b)\\
    & \geq R
\end{align*}

Further, if $g \in S$, then \begin{align*}
    d((y, a), g(y, a)) &= d(y, gy) + d(a, ga) \\
    & \leq B+ \e\\
    & = \e'
\end{align*}

Similarly, $d((y,b), g(y,b)) \leq \e'$. By acylindricity of the action on $Y \times Z$, we have that $|S| \leq \mfN$.    \end{proof}

Similar to the previous result, the next lemma allows us to remove parabolic factors from an acylindrical action on a product. Again, we prove the result in a more general setting, as it might be helpful for future explorations.

\begin{prop}\label{prop:removepara} Let $\G \to \Isom Z \times \Isom X$ be acylindrical, where $Z$ is a metric space and $X$ is hyperbolic. If the projection $\G \to\Isom X$ is parabolic, then the projection $\G \to \Isom Z$ is acylindrical. 
 \end{prop} 

\begin{proof} Since the action of $\G$ on $X$ is parabolic, $\G$ has a unique fixed point $\xi$ in $\partial X$, and hence the associated Busemann quasimorphism $\b: \G\to \R$ is identically $0$. Let $E$ be the constant from Lemma~\ref{Lem:shiftfn} and $q$ a $(1, 20\delta)$ quasigeodesic converging to $\xi$. 

We proceed by contrapositive. Assume, that the action on $Z$ is not acylindrical. Then, there exists an $\e >0$ such that for any $R, \mfN >0$, there exist points $z,w\in Z$ with $d(z,w) \geq R$ and $$| \set_\e\{z,w\}| =| \{ g \in \G \mid d(z, gz) \leq \e \text{ and } d(w, gw) \leq \e \}| \geq \mfN.$$

Set $\e' = \e + E$. We will complete the proof by showing that the action of $\G$ on $Z \times Y$ is not acylindrical for this $\e' >0$.  Let $\G=\Cup{n\geq 1}{}S_{n}$ be a nested sequence of finite symmetric sets exhausting $\G$ (the group is infinite as the orbit is unbounded). Choose $m$ large enough so that $|\set_\e\{z,w\}\cap S_m| \geq N$. By Lemma~\ref{lemmaelimpara}, there is a $t_m>0$ such that $\mathrm{M}(t) \geq m$ for all $t \geq t_m$. It follows from the proof of Lemma~\ref{lemmaelimpara}, that $y = q(t_m) \in X$ is a point such that $$d(gy, y) \leq E$$ for all $g \in S_m \cap \set_\e\{z,w\}$. Then $(z, y), (w,y) \in Z \times X$ and $d((z, y), (w,y)) \geq R.$ Consider $g \in S_m \cap \set_\e\{z,w\}$. Then for $u\in \{z,w\}$, we have that $$d((u,y), g(u, y)) = \max{} \{d(u, gu), d(y, gy)\} \leq \e + E \leq \e'.$$ As a consequence $| \set_{\e'} \{(z, y), (w, y)\}| \geq \mfN$. This shows that any choice of $R, \mfN$ fail to satisfy the acylindricity condition for $\e'$ on $Z \times X$, as claimed.
\end{proof}

\begin{cor}\label{cor:elimpapra} Let $\G \to \Aut_{\!0}\X$ such that each factor is parabolic. Then the action of $\G$ on $\X$ is not acylindrical. 
\end{cor}

\begin{proof} Assume, by contradiction, that the action on $\X$ is acylindrical. Repeatedly applying Proposition~\ref{prop:removepara} to the factors of $\X$ ultimately leaves one factor with a parabolic, acylindrical action, which is impossible by \cite[Theorem 1.1]{Acylhyp}.
\end{proof}

We now eliminate the possibility of quasiparabolic actions on all factors in the case of acylindrical actions in higher rank. 

\begin{cor}\label{cor: clean up acyl} Let $\rho:\G\to \Aut_{\! 0}\X$ be acylindrical and not elliptic. Let $I\subset\{1, \dots, \D\}$ be the subset of indices such that the factor action $\rho_i: \G\to \Isom X_i$ is neither elliptic, nor parabolic. Then $I\neq \varnothing$ and $\Delta(\rho_i: i\in I): \G\to \Prod{i\in I}{}X_i$ is acylindrical. 
\end{cor}

\begin{proof}
   Let $I_{ell},I_{par}\subset \{1, \dots, \D\}$ be the indices for which $\rho_i(\G)$ is elliptic, respectively parabolic. Since $\rho(\G)$ is not elliptic, we have that $J:=\{1, \dots, \D\}\setminus I_{ell}$ is not empty. By Lemma~\ref{Lem:elimellfactors}, the action $\Delta(\rho_i: i\in J): \G\to \Prod{i\in J}{}X_i$ is acylindrical. 

By Corollary~\ref{cor:elimpapra}, we have that $I= J\setminus I_{par}$ is not empty. And by Proposition~\ref{prop:removepara}, we have that $\Delta(\rho_i: i\in I): \G\to \Prod{i\in I}{}X_i$ is acylindrical. 
\end{proof}

\begin{lemma}\label{Lem:noqpinacyl} Suppose that $\G \to \Aut\X$ is acylindrical and fixes some $\xi \in \partial_{reg} \X$. Then every finitely generated subsemigroup of $\G$ has polynomial growth and no factor is quasiparabolic.  
\end{lemma}

\begin{proof} Since acylindricity of the action passes to subgroups, without loss of generality, up to passing to a finite index subgroup, we may assume that $\G\to \Aut_{\!0}\X$. The components of $\xi$ are  $\xi_i\in \partial X_i$  $i\in \{1, \dots, \D\}$ and are $\G$-fixed. For each $i\in \{1, \dots, \D\}$, let  $q_i: \N \to X_i$  be a $(1, 10\delta)$ quasigeodesic ray converging to $\xi_i$, and  let $\b_i: \G\to \R$ be the Busemann quasimorphism associated to the action in each factor and of defect at most $B \geq 0$.  Set $\b: \G\to \R^{\D}$ and $q: \N \to \X$ to be the associated diagonal maps. i.e. $\b(g) = \Delta(\b_1(g),\dots, \b_\D(g) )$ and $q(n) = \Delta(q_1(n),\dots, q_\D(n) )$.  Since we are using the $\ell^\infty$-product metric on $\X$ and $\R^{\D}$, $q$ is a $(1, 20\delta)$ quasigeodesic in $\X$.

Define $F_n = \left\{ g \in \G \: \|\b(g)\|_\8 \leq n \right\} \subset \G$ for all $n\in \N$. The reader may think of these as F\o lner sets and we will show they have polynomial growth. 

\noindent
\textbf{Claim 1:} For each $n \in \N$, $| F_n| \leq \mfN$, where $\mfN$ is the smallest constant from the definition of acylindricity associated to $\e_n = n+ 20\delta + E$, where $E$ is the constant from Lemma~\ref{Lem:shiftfn}. 

Suppose by contradiction, that there exists a $k_0$ such that $|F_{k_0}| > \mfN$, for the $\mfN$ associated to $\e_{k_0}$. Let $Q\subset F_{k_0}$ be of cardinality  $\mfN+1$. Recall our notation that we have $h\mapsto (h_1, \dots, h_{\D})$. For each $ h\in Q$, it follows from Lemma~\ref{Lem:shiftfn}, that if $t \geq t_0(h_i)$, then $$d(q_i( t - \b_i(h)), h_iq_i(t)) \leq E.$$ Using the triangle inequality and that $q_i$ is a $(1, 20\delta)$ quasigeodesic, we get 
\begin{align*}
    d(q_i(t), h_iq_i(t)) &\leq d(q_i(t), q_i(t - \b_i(h)) + d(q_i(t - \b(h), h_iq_i(t))\\
    & \leq |\b_i(h)| + 20\delta + E\\
    & \leq k_0 + 20\delta +E = \e_{k_0}.
    \end{align*}

Take $T =\op{max} \{ t_0(h_i) \mid h \in Q, 1 \leq i \leq \D \}$. Then for all $t \geq T$, it follows that $$d(q(t), h.q(t)) \leq \e_{k_0}$$

Choose $a, b \geq T$ such that $d(q(a), q(b)) \geq R(\e)$, where $R$ is the constant from the definition of acylindricity for $\e_{k_0}$. Such $a,b$ exists as $q(t)$ is an infinite quasigeodesic. Then it follows that for all $h \in Q$, $d(q(a), h.q(a)) \le\e_{k_0} \text{ and } d(q(b), h.q(b)) \leq \e_{k_0}$. This contradicts the acylindricty of the action as since $|Q| = \mfN +1 > \mfN$. Thus Claim 1 holds. 

It is easy to verify that $\bigcup_{n} F_n = \G$.  Since we use subscripts to denote the factors of $\X$, we will use superscripts to enumerate elements in the following step -- these should not be confused for powers. Let $s^1, \dots, s^k \in \G$ and let $S=\<s^1, \dots, s^k \>^+$ be the semigroup generated by them. We now show that  $S$ has polynomial growth. Fix $m \in \N$ to be the smallest index such that $s^1,\dots, s^k \in F_m$.

\noindent
\textbf{Claim 2:} For any $l$, the product $a^1\cdots a^l \in F_{l(m + B)}$ where each $a^i \in \{s^1, \dots, s^k\}$ is a generator of $S$.

Since $a^1, \dots, a^l \in F_k$ and $\b_i$ is a quasimorphism, it is easy to check that for each $i= 1, \dots, {\D}$, $$|\b_i(a^1\cdots a^l)| \leq lB + |\b_i(a^{1})| + \dots |\b_i(a^{l})| \leq lB + lm,$$ which proves the claim.   

By  Theorem~\ref{thm:actionsclassrk1},  no factor can be quasiparabolic as this would yield a finitely generated subsemigroup of $\G$ with exponential growth.
\end{proof}

We now have all the tools to prove the following theorem.

\begin{theorem}\label{thm:elimmixedfactors} Let $\G \to \Aut_{\!0}\X$  such that the projections $\G\to \Isom X_i$ are all either elliptic, parabolic or quasiparabolic (with at least one factor being parabolic or quasiparabolic). Then $\G \to \Aut \X$ is not acylindrical.

In particular, if $\G \to \Aut_{\!0}\X$ is acylindrical, then every element of $\G$ is either an  elliptic, weakly loxodromic, or regular loxodromic isometry.
\end{theorem}

\begin{proof} The theorem follows by combining Lemmas~\ref{Lem:elimellfactors} and \ref{Lem:noqpinacyl} and Proposition~\ref{prop:removepara}. If $\G \leq \Aut_{\!0}\X$ is acylindrical, then so is every subgroup, in particular $\langle g \rangle$ for each $g \in \G$. The result now follows since it must be the case that $g$ is either elliptic in each factor or has at least one loxodromic factor.
\end{proof}

\subsection{Regular Elements Exist}
We now turn our attention to regular elements i.e those that preserve factors and act as a loxodromic in each factor. Regular elements will be consider more in Section~\ref{sec:elemsubgrp}. Our goal for this subsection is to prove the following proposition. (See also \cite{ CapraceSageev, CapraceZadnik, FLM}.)

\begin{prop}\label{Prop:Regular Exist}
If  $\G \to \Aut\X$ has general type factors then $\G$ contains regular elements.
\end{prop}

The proof of Proposition~\ref{Prop:Regular Exist} relies on the following results of Maher and Tiozzo and properties of probability measures, which will allow us to find regular elements in actions with general type factors with asymptotically high probability.

\begin{theorem}\cite[Theorem 1.4]{MaherTiozzo}\label{thm: MaherTiozzo}
 Let $\G\to \Isom X$ be a general type action, and $\mu$ be a generating symmetric probability measure on $\G$. Then the translation-length $\tau(w_n)$ of the $n^{th}$-step of the $\mu$-random walk grows at least linearly, i.e. there exists $T>0$ such that $\mathbb P(w\in \G^\N:\tau(w_n)\leq Tn) \to 0$ as $n\to \8$. 
\end{theorem}

\begin{lemma}\label{Lem:prob}
Let $\e\in (0,1)$, and $\P$ be a probability measure on $\Omega$. Assume $A_1, \dots, A_{\D}\subset \Omega$ are measurable sets such that $\P\left(A_i\right) > 1-\frac{\e}{\D}$ for each $i=1, \dots, {\D}$. Then $\P\left(\Cap{i=1}{\D}A_i\right) >1-\e$. 
\end{lemma}

\begin{proof}
We shall use $A^c:= \Omega \setminus A$ to denote set compliments. For $i=1, \dots, {\D}$ our assumption is equivalent to $\P(A_i^c) < \frac{\e}{\D}$ which means that 
$$ \P\left(\Cup{i=1}{\D}A_i^c\right)\leq\Sum{i=1}{\D}\,\P(A_i^c) < \e$$
Therefore, by De-Morgan's Law, $\P\left(\Cap{i=1}{\D}A_i\right)> 1-\e$.
\end{proof}

\begin{proof}[Proof of Proposition~\ref{Prop:Regular Exist}]
As before, up to passing to a finite index subgroup, without loss of generality we  assume that $\G\to \Aut_{\!0}\X$. Fix a generating probability measure $\mu$ on $\G$, and $0<\e<1$. Since each projection is of general type, we may apply Theorem~\ref{thm: MaherTiozzo} to each action $\G\to \Isom X_i$, where $\tau_i: \G \to \R$ denotes the stable translation length of elements acting on $X_i$. This gives real numbers $T_i>0$, so that for any $n$ sufficiently large we have the following measurable sets:

\begin{enumerate}
    \item $A_i = \{w\in \G^\N:\tau_i(w)> T_i n\}$
    \item $\P(A_i) > 1 - \dfrac{\e}{\D}$
    
\end{enumerate}

It then follows from Lemma~\ref{Lem:prob} that $\P\left(\Cap{i=1}{\D}A_i\right) > 1 -\e$. In particular, the set $\Cap{i=1}{\D}A_i$ is non-empty and consists of regular elements.
\end{proof}

\subsection{Some Representation Theory}\label{Sec: rep theory basics}

Recall that an action by a group $\G$ on a set $S$ corresponds to a homomorphism from $\G$ to the group of permutations of $S$. If the action preserves additional structure, then the image of said homomorphism will belong to the subgroup of the  permutation group that preserves that structure. Therefore, when considering actions, representation theortic tools are indispensable. In this section, we explore the interactions between (AU-)acylindricity in the higher rank setting with some representation theoretic tools. Again, we use the $\ell^\8$-product metric on the product spaces. This means that if the distance in the product is large then some factor must be large. Similarly, if the distance in the product is small then all factors are small.

We begin with diagonal actions. Recall from the universal property for products, given homomorphisms $\rho_i: \G\to \Isom Y_i$, for $i=1, \dots, \D$, there is a unique diagonal homomorphism $\rho = \Delta(\rho_1, \dots, \rho_\D): \G\to \Prod{i=1}{\D}\Isom Y_i$ determined by the property that post-composition with the natural projection $\pi_j: \Prod{i=1}{\D}\Isom Y_i\to \Isom Y_j$ yields the equality $\pi_j\circ \rho= \rho_j$ for each $j=1, \dots, \D$.

\begin{lemma}\label{lem: diag of acyl is acyl} Let $\D\in \N$, $\G$ be a group, $Y_i$ be metric spaces. Suppose that $\rho_i: \G\to \Isom Y_i$ is an (AU-)acylindrical action for each $i=1, \dots, \D$. Then, the diagonal action $\rho=\Delta(\rho_1, \dots, \rho_\D):\G\to \Prod{i=1}{\D}\Isom Y_i$ is  (AU-)acylindrical. 
\end{lemma}

\begin{proof} Let $\mathbb{Y} =\Prod{i=1}{\D} Y_i$, and let us use the $\ell^\8$-product metric. Consider an element $g \in \G$ and points $x,y\in\mathbb{Y}$. Let $\e > 0$ be given.  Since $\G \to \Isom Y_i$ is acylindrical for all $i$, choose  $R, N$ to be the maximum of the associated acylindricity constants  for $\e$ from each factor $Y_i$. We will show that $R, N$ satisfy acylindricity requirements for the diagonal action of $\G$ on $\mathbb Y$. 

Let $\displaystyle d(x, y) = \max{i =1, \dots,\D} d(x_i , y_i) \geq R$. Then $d(x_j , y_j) \geq R$ for some $1 \leq j \leq \D$.  Let $g \in \set_\e\{x,y\}$. Since $d( x, \rho(g) x)$, respectively $d( y, \rho(g) y)$ is the maximum of the coordinate distances, the inequality holds in the $j^\textrm{th}$ coordinate, i.e.  $d(x_j, \rho_j(g)x_j) \leq \e$ and similarly, $d(y_j, \rho_j(g)y_j) \leq \e$. This shows that $\set_\e\{x,y\}\subset \set_\e\{x_j,y_j\}$, and hence
  $$|\set_\e\{x,y\}|\leq |\set_\e\{x_j,y_j\}|\leq N.$$
\end{proof}

\begin{lemma}\label{lem: duplicate action}
    Let $Y$ and $Z$ be metric spaces. Consider homomorphisms $\rho: \G\to \Isom Y$, $\sigma : \G\to \Isom Z$. For an element $\a\in \Isom Y$, let $\rho_\a(g) = \a\rho(g)\a^{-1}$. The associated diagonal action $\Delta(\rho, \sigma): \G\to \Isom Y \times \Isom Z$ is (AU-)acylindrical if and only if the diagonal action $\Delta(\rho,\rho_\a, \sigma): \G\to \Isom Y \times\Isom Y \times \Isom Z$ is (AU-)acylindrical.
\end{lemma}

This lemma shows that duplicating a representation neither helps nor hinders (AU-) acylindricity, even up to postcomposition with an inner-automorphism of $\Isom Y$.

\begin{proof} We shall prove the case of acylindricity since the case of ambiguous uniformity is similar.  
For clarity we write down the actions of $\G$ on $Y\times Z$ and $Y\times Y\times Z$, without the use of the cumbersome notation $\Delta(\rho, \sigma)$ and $\Delta(\rho,\rho_\a, \sigma)$: If
$g\in \G$, $(y,z)\in Y\times Z$,  $(y,y',z)\in Y\times Y\times Z$ then 
$$g.(y,z)= \big(\rho(g)y,\sigma(g)z\big)\;\;\text{ and } \;\; g.(y,y', z)= \Big(\rho(g)y,\big(\a\rho(g)\a^{-1}\big)y,\sigma(g)z\Big).$$

We will need the following calculation. We  utilize a blackboard comma in various places in hopes of aiding readability:
\begin{eqnarray}
\begin{split}
    d\bigg(g.(y,y', z)\bbcomma (b,b', c)\bigg) 
    &= d\bigg(\Big(\rho(g)y,\big(\a\rho(g)\a^{-1}\big)y,\sigma(g)z\Big)\bbcomma \Big(b,b', c\Big)\bigg)
    \\ &= \max{}\left\{d\Big(\rho(g)y,b\Big)\bbcomma d\Big(\big(\a\rho(g)\a^{-1}\big).y',b'\Big)\bbcomma d\Big(\sigma(g)z,c\Big)\right\}
    \\ &= \max{}\left\{d\Big(\rho(g)y,b\Big)\bbcomma d\Big(\rho(g)\big(\a^{-1}.y'\big),\a^{-1}.b'\Big)\bbcomma d\Big(\sigma(g)(z),c\Big)\right\}
    \\ &= d\bigg(\Big(\rho(g)y, \rho(g)\big(\a^{-1}.y'\big), \sigma(g)(z)\Big)\bbcomma \Big(b, \a^{-1}.b', c\Big)\bigg).
    \end{split}
\end{eqnarray}

 A similar calculation for $d\big(g.(y,z)\bbcomma (b,c)\big)$ shows that 
\begin{equation}\label{eq: 3 to 2}
    d\Big(g.(y, \a.y,z)\bbcomma (b, \a.b, c)\Big)=d\Big(g.(y,z)\bbcomma (b,c)\Big).
\end{equation}
It follows that 
\begin{equation}\label{eq: set up.down}
\begin{split}
 \set_\e\{(y,y', z)\} &=\set_\e\{(y, z)\}\cap \set_\e\{(\a.y', z)\}\\
   &= \{g\in \G: d(\rho(g)y,y)\bbcomma d(\rho(g)\a.y',\a.y')\bbcomma d(\sigma(g)z,z)\leq \e\}.
\end{split}
\end{equation}
These in turn allow us to calculate that 
\begin{eqnarray}\label{eq: triple down}
\begin{split}
   \set_\e\{(y,y',z)\bbcomma (b,b',c)\} & = \set_\e\{(y,y',z)\} \cap \set_\e\{ (b,b',c)\} \\
   & = \set_\e\{(y,z)\} \cap \set_\e\{(\a.y',z)\}\cap \set_\e\{(b,c)\} \cap \set_\e\{(\a.b',c)\} \\
   & = \Big(\set_\e\{(y,z)\bbcomma(b,c)\}\Big) \cap \Big(\set_\e\{(\a.y',z)\bbcomma(\a.b',c)\} \Big)
    \end{split}
\end{eqnarray}

Suppose now that the action of $\G$ on $Y\times Y\times Z$ is acylindrical. Fix $\e>0$ and let $R, N\geq 0$ be the associated constants. We claim that these constants work for the action of $\G$ on $Y\times Z$ as well. 

Let $(y,z)\bbcomma (b,c)\in Y\times Z$ be such that $d((y,z)\bbcomma (b,c))\geq R$. By equation (\ref{eq: 3 to 2}) (for $g=1$) we have that $d((y,\a.y,z)\bbcomma (b,\a.b,c))\geq R$. Furthermore, we claim that $\set_\e\{(y,z)\bbcomma (b,c)\}= \set_\e\{(y,\a.y,z)\bbcomma (b,\a.b,c)\}$. By definition, $\set_\e\{(y,z)\bbcomma (b,c)\} = \set_\e\{(y,z)\}\cap \set_\e\{(b,c)\}$. By equation  (\ref{eq: 3 to 2}), the latter is equal to $\set_\e\{(y, \a. y, z)\}\cap \set_\e\{(b, \a. b, c)\}$, which by definition is $\set_\e\{(y,\a.y,z)\bbcomma (b,\a.b,c)\}$.

 Therefore, since $\set_\e\{(y,z)\bbcomma (b,c)\}= \set_\e\{(y,\a.y,z)\bbcomma (b,\a.b,c)\}$ we deduce that $$|\set_\e\{(y,z) \bbcomma (b,c)\}|\leq N.$$

Conversely, suppose that the action of $\G$ on $Y\times Z$ is acylindrical. Fix $\e>0$ and let $R, N\geq 0$ be the associated constants. We claim that these constants work for the action of $\G$ on $Y\times Y\times Z$ as well. 

Let $(y,y',z)\bbcomma (b,b',c)\in Y\times Y\times Z$ be such that $d((y,y',z)\bbcomma (b,b',c))\geq R$. By calculation (\ref{eq: triple down}) above,  we have that $$\set_\e\{(y,y',z)\bbcomma (b,b',c)\} = \set_\e\{(y,z)\bbcomma(b,c)\}\cap \set_\e\{(\a.y',z)\bbcomma(\a.b',c)\} .$$ Furthermore, if $d((y,y',z)\bbcomma (b,b',c))>R$ then one of the factors must be large, i.e.  $d(y,b)\geq R$, $d(y',b')\geq R$ or $d(z, c)\geq R$. Either way, by acylindricity of the action of $\G$ on $Y\times Z$, the coarse stabilizer has cardinality at most $N$.
 \end{proof}

In what follows, the homomorphism $\rho'$ is the induced representation from $\G$ to $\G'$ and often denoted by $\mathrm{ind}_{\G}^{\G'}\!\rho$. This somewhat cumbersome notation is sometimes useful, but it will not be for our purposes (see for example \cite[Appendix E]{BdlHV}.) Breaking with our previous convention of how elements in a product are written, given a finite set $F$, we identify $\X^{|F|}$ with the set of sequences indexed by elements in $F$ and values in $\X$, i.e. for each $\chi\in F$ and choice of $\mathbf{x}_\chi\in\X$, the corresponding element in $\Prod{\chi\in F}{}\X$ will be denoted by $(\mathbf{x}_\chi)_{\chi\in F}$. The formulation below comes from \cite[Example 4.2.12]{Zimmer} and is also used in  \cite[page 1793]{FernosRelT}. We note that the induction can be performed generally from any action of $\G$. We will need the following lemma, which is straightforward to check. 

\begin{lemma}\label{lem: precomp aut and acyl} Let $\G$ be a  group, $\a\in \Aut \G$ an automorphism, $Y$ a metric space and  $\rho: \G\to \Isom Y$. The action $\rho(\G)$ is (AU-)acylindrical if and only if the action $\rho\circ\a: \G\to \Isom Y$ is (AU-)acylindrical. 
\end{lemma}

\begin{prop}[Induced representation]\label{prop: fin ind induction}
     Let $\G\leq \G'$ be of finite index and let $F= \G'/\G$ be the set of cosets, where $\chi_0\in F$ is the identity coset. If $\rho:\G\to \Aut\X$ is a homomorphism then there exists $\rho':\G'\to \Aut\! \big(\Prod{\chi\in F}{}\X\big)=\Sym(F)\ltimes \Prod{\chi\in F}{}\Aut \X$ such that
     \begin{enumerate}
         \item\label{(1)} the permutations in $\Sym (F)$ associated to $\rho'(\G)$ fix the identity coset $\chi_0\in F$ and the associated projection coincides with $\rho$, i.e. $\pi_{\chi_0}\circ\rho'= \rho: \G\to \Aut\X$;
         \item\label{(2)} if $\G_0\leq \G$ is the kernel of the permutation representation $\G' \to \Sym(F)$ then for each $\kappa\in F$ there exists an inner automorphism $\a_\kappa\in \Inn(\G')$ such that $\rho'|_{\G_0} = \Delta(\rho|_{\G_0}\circ \a_\chi)_{\chi\in F}$;
         \item\label{(3)} $\rho(\G)$ has general type factors if and only if $\rho'(\G')$ has general type factors;
         \item\label{(4)} if the action of $\rho(\G)$ on $\X$ is (AU-)acylindrical then the action of $\rho'(\G')$ on $\Prod{\chi\in F}{}\X$  is (AU-)acylindrical.
     \end{enumerate}
\end{prop}

The maps are summarized in the following commutative diagrams, where the multiple arrows on the right represent the various maps according to $\kappa\in F$ (see Section~\ref{subsec: Univ Prop}):

\begin{center}
 \begin{tikzcd}
		{\G'} & {\Sym(F)\ltimes\Prod{\chi\in F}{}\Aut\X} \\
	& {\Sym(F\setminus\{\chi_0\})\ltimes(\Prod{\chi\in F}{}\Aut\X)} \\
	\G & {\Aut \X}
	\arrow["{{\rho'}}", from=1-1, to=1-2]
	\arrow[hook, from=2-2, to=1-2]
	\arrow["{{\pi_{\chi_0}}}", two heads, from=2-2, to=3-2]
	\arrow[hook, from=3-1, to=1-1]
	\arrow["{{\rho'|_{\G}}}", squiggly, from=3-1, to=2-2]
	\arrow["\rho"', from=3-1, to=3-2]
\end{tikzcd}
\begin{tikzcd}
	{\G'} && {\Sym(F)\ltimes\Prod{\chi\in F}{}\Aut\X} \\
	{\G_0} && {\Prod{\chi\in F}{}\Aut\X} \\
	\\
	{\G_0} && {\Aut \X}
	\arrow["{{{{\rho'}}}}", from=1-1, to=1-3]
	\arrow[hook, from=2-1, to=1-1]
	\arrow["{{\rho'|_{\G_0}}}", squiggly, from=2-1, to=2-3]
	\arrow[hook, from=2-3, to=1-3]
	\arrow["{{\pi_\kappa, \kappa\in F}}", two heads, from=2-3, to=4-3]
	\arrow[shift right, two heads, from=2-3, to=4-3]
	\arrow[shift right=2, two heads, from=2-3, to=4-3]
	\arrow["{{\a_\kappa, \kappa \in F}}", shift left=2, from=4-1, to=2-1]
	\arrow[shift left, from=4-1, to=2-1]
	\arrow[from=4-1, to=2-1]
	\arrow["{{\Delta(\rho_{\G_0}\circ\a_\chi)_{\chi\in F}}}"{description}, dashed, from=4-1, to=2-3]
	\arrow[from=4-1, to=4-3]
	\arrow[shift left, from=4-1, to=4-3]
	\arrow["{{\rho|_{\G_0}\circ\a_\kappa, \kappa\in F}}"', shift right, from=4-1, to=4-3]
\end{tikzcd}
\end{center}
\begin{proof} 
Let $\X'= \Prod{\chi\in \F}{}\X= \{(\mathbf{x}_\chi)_{\chi\in F}: \mathbf{x}_\chi\in \X\}$. We begin by defining the homomorphism $\rho': \G'\to \Aut \X'=\Sym(F)\ltimes \Prod{\chi\in F}{}\Aut \X$.
Recall that since $F=\G'/\G$ is the set of left cosets, left multiplication in $\G'$ yields a natural permutation representation $\G'\to \Sym(F)$. Let $p:\G'\to \G'/\G$, given by $p(g) = g\G=g\chi_0$ be the natural projection. Choose a section $s:F\to \G'$ taking the identity coset to the identity of $\G'$, i.e. $p(s(\chi)) = \chi$ and $s(\chi_0)=1$. Consider the map $c: \G'\times F\to \G$  given by $$c(g,\chi)= [s(g \chi)]^{-1}gs(\chi).$$
The fact that $c$ is well defined follows from the fact that $p(gs(\chi))=p(s(g\chi))$. Further, it is easy to verify that the map $c$ is a cocycle, meaning that it satisfies the relation:
$$c(g_1g_2,\chi) = c(g_1, g_2\chi)\cdot c(g_2, \chi).$$

   We are now ready to define the induced homomorphism $\rho': \G'\to \Aut\X'$. For $g\in \G'$ define
$$\rho'(g)\Big((\mathbf{x}_\chi)_{\chi\in F}\Big) := \Big(\rho\big(c(g,\chi)\big)(\mathbf{x}_\chi)\Big)_{g\chi\in F}$$

This definition has a permutation of the cosets according to $g: \chi\mapsto g\chi$ (witnessed in the index) and in each factor we have $\rho(c(g,\chi))\in \Aut \X$. Therefore, $\rho'(g)\in \Sym(F)\ltimes \Prod{\chi\in F}{}\Aut\X$ as promised and the cocycle relation ensures this is indeed a homomorphism.

\noindent
\textbf{Item (\ref{(1)}):} Observe that if $g\in \G$ then $g\G=\G$. Since $\G=\chi_0$, this means that $g\chi_0=\chi_0$. Also, we have chosen  $s(\chi_0)=1$, and so $c(g,\chi_0)= s(g\chi_0)^{-1}gs(\chi_0)= g$, i.e. $c(\cdot, \chi_0)|_{\G}:\G\to \G$ is the identity map. This means that $\rho'(\G)$ preserves the $\chi_0$-coordinate. We may therefore postcompose with the associated projection $\pi_{\chi_0}: \Sym(F\setminus\{\chi_0\})\ltimes \Prod{\chi\in F}{}\Aut\X\to \Aut\X$ and the above calculation shows that indeed $\pi_{\chi_0}\circ \rho'|_{\G}=\rho: \G\to \Aut \X$. This completes the proof of the first claim in the statement.

\noindent
\textbf{Item (\ref{(2)}):}
Let $\G_0\leq \G$ denote  the finite index kernel of the permutation action $\G'\to \Sym(F)$. (Note that, in case $\G$ was normal to begin with, we would have that $\G_0=\G$.) 
If $g\in \G_0$ then $g\chi=\chi$ for each $\chi\in F$, meaning that 
$$c(g, \chi)= s(g\chi)^{-1}gs(\chi) = s(\chi)^{-1}gs(\chi).$$
In other words, $c(\cdot, \chi)|_{\G_0}: \G_0\to \G_0$ is the inner automorphism $\a_\chi:g\mapsto s(\chi)^{-1}gs(\chi)$. Furthermore, 
\begin{eqnarray*}
    \rho'(g)((\mathbf{x}_\chi)_{\chi\in F}) &=& (\rho(c(g,\chi))(\mathbf{x}_\chi))_{g\chi\in F}\\
     &=& (\rho\circ\a_\chi(g)(\mathbf{x}_\chi))_{\chi\in F},
\end{eqnarray*}
meaning that $\rho'|_{\G_0} = \Delta(\rho\circ \a_\chi)_{\chi\in F}.$

\noindent
\textbf{Item (\ref{(3)}):}
The factor actions of $\rho'(\G')$ are the factor actions of $\rho(\G)$ together with their precompositions with $\a_\chi$, for $\chi\in F$, which preserves the action being general type.

\noindent
\textbf{Item (\ref{(4)}):}
We shall prove the claim for acylindricity; as before the nonuniform case follows analogously. Furthermore, since  acylindricity is inherited by subgroups, it is sufficient to prove that if $\rho(\G_0)$ is acylindrical then $\rho'(\G')$ is acylindrical. As established in (\ref{(2)}), $\rho'_{\G_0} = \Delta(\rho\circ\a_\chi)_{\chi\in F}$, which is acylindrical by Lemma~\ref{lem: precomp aut and acyl}.

Let $C\subset \G'$ be a set of coset representatives for $\G_0\leq \G'$. Then $\G'=\Cup{h\in C}{}h\G_0$. We claim that 
$$\set_\e^{\G'}\{(\mathbf{x}_\chi)_{\chi\in F},(\mathbf{y}_\chi)_{\chi\in F}\} \subset\Cup{h\in C}{}\Big[\Big( \set_\e^{\G_0}\big\{\rho'(h) (\mathbf{x}_\chi)_{\chi\in F},\rho'(h)(\mathbf{y}_\chi)_{\chi\in F}\big\}\Big)\cdot h\Big].$$

Indeed, if $g\in \set_\e^{\G'}\{(\mathbf{x}_\chi)_{\chi\in F},(\mathbf{y}_\chi)_{\chi\in F}\} \subset \G'$ then there exists $h\in C$ such that $g=hg_0$ for some $g_0\in \G_0$, and so

\begin{eqnarray*}
    d\big(\rho'(g) (\mathbf{x}_\chi)_{\chi\in F},\rho'(g)(\mathbf{y}_\chi)_{\chi\in F}\big) &=& d\big(\rho'(hg_0) (\mathbf{x}_\chi)_{\chi\in F},\rho'(hg_0)(\mathbf{y}_\chi)_{\chi\in F}\big) \\
    &=& d\Big(\rho'(hg_0h^{-1})\rho'(h) (\mathbf{x}_\chi)_{\chi\in F},\rho'(hg_0h^{-1})\rho'(h)(\mathbf{y}_\chi)_{\chi\in F}\Big) 
\end{eqnarray*}

Since $\G_0$ is normal, we have  that $hg_0h^{-1}\in \set_\e^{\G_0}\big\{\rho'(h) (\mathbf{x}_\chi)_{\chi\in F},\rho'(h)(\mathbf{y}_\chi)_{\chi\in F}\big\}$, and so 

$$hg_0\in \Big(\set_\e^{\G_0}\big\{\rho'(h) (\mathbf{x}_\chi)_{\chi\in F},\rho'(h)(\mathbf{y}_\chi)_{\chi\in F}\big\}\Big)\cdot h$$

Let $\e>0$ and $R,N_0\geq 0$ be the acylindricity constants for $\rho'(\G_0)$. The above calculations show that $R$ and $N=|F|\cdot N_0$ are acylindricity constants for $\rho'(\G')$. 
\end{proof}

\section{A Plethora of Examples}\label{sec:examples_p1}

Let $G_1, \dots, G_F$ be locally compact second countable groups. A subgroup $\G\leq \Prod{i=1}{F}G_i$ is said to be a \emph{lattice} if it is discrete and has finite covolume (with respect to the Haar measure on the product). Moreover, it is said to be an \emph{irreducible lattice} if the natural projection of $\G$ to $\Prod{j\in J}{}G_j$ is dense (or dense in a finite index subgroup) for all $J\subsetneq \{1,\dots, F\}$.  In both the reducible or irreducible case, discreteness ensures that $\G$ acts properly on $G = \Prod{i=1}{F}G_i$. Moreover, if the factors of $G$ have a nice model geometry then $\G$ acts properly on the associated product as well. For example,  if $\G$ is ($S$-)arithmetic semi-simple and $K_i\leq G_i$ is maximal compact then  $\Prod{i=1}{F}G_i/K_i$ is a product of symmetric spaces (and Bruhat-Tits buildings when $S\neq \varnothing$), each of which is CAT(0). Taking the $\ell^2$-product metric yields a CAT(0) space on which $\G$ acts properly and hence AU-acylindrically. 

In the same way that $\delta$-hyperbolic spaces ``coarsify" the notion of negative curvature, we consider finite products of $\delta$-hyperbolic spaces with the $\ell^2$-product metric as a coarse version of nonpositive curvature (though admittedly we almost exclusively work with the Lipschitz equivalent $\ell^\8$-product metric). Similarly one can view groups acting AU-acylindrically on said products as a coarsening of a lattice in nonpositive curvature, or a ``generalized lattice". We  will further this analogy in \cite{BFPart2}: if all factor actions are of general type, then the generalized lattice enjoys a type of semi-simple behavior.

In this section we will discuss the class of groups which we can considered to be generalized lattices within the setting of finite products of $\delta$-hyperbolic spaces.
It is straightforward to verify that the class of groups that admit such an AU-acylindrical (but not proper) action is not closed under direct products. Nevertheless, we prove in \cite{BFPart2} that the consequences we deduce of said class remain true for groups that  can be represented as a subdirect product in a product of groups within the class.

Additionally, we also prove in \cite{BFPart2} that our class of groups is closed under passing to ``general type" subgroups  and virtual isomorphism. Interestingly, to date there are two notable related questions that remain open. Namely whether the class of CAT(0) groups (respectively acylindrically hyperbolic groups) is invariant under virtual isomorphism, in particular finite index extension. The key issue here is that the standard approach of inducing a representation from a finite index subgroup (see Proposition~\ref{prop: fin ind induction}) does not produce a space sufficiently small to retain a cocompact action (respectively retain hyperbolicity). The class under consideration here however, is not mired by these obstacles, thanks to the intertwining of the  (AU-)acylindricity property and the higher-rank structure.

We compile a list of examples, beginning with those that can be found in the existing literature.

\begin{enumerate}
\item Acylindrically hyperbolic groups (see also \cite[Theorem 1.2]{Acylhyp}), with $\D=1$. This is, in itself, an extensive class of groups, including, for example:

\begin{itemize}
\item Most (non-finite) mapping class groups \cite{Bowditch2008, MasurMinsky}.
    \item Non-cyclic Hyperbolic groups
    \item Relatively hyperbolic groups (see e.g. \cite{DGO})
     \item Nonamenable amalgamated products (see \cite{MinasyanOsin} and \cite[Theorem 2.25]{DGO})
    \item Many small cancellation groups \cite{GruberSisto}
    \item  Groups acting properly on proper CAT(0) spaces with rank-one elements \cite{Sisto2018}. In particular, this includes irreducible RAAGs. (See also \cite{ChatterjiMartin, Genevois}.)
    \item Many Artin groups \cite{HaettelXXLArtin, MartinPrzytycki, Vaskou, HagenMartinSisto}

\end{itemize}

\item Since our class is closed under passing to general type subgroups that act properly, we also obtain groups with exotic finiteness properties such as the kernels of Bieri-Stallings and Bestvina-Brady, see \cite{Bieri, Stallings}, \cite{BestvinaBrady} and other subdirect products such as those in \cite{BridsonHowieMillerShort, Bridson, BridsonMiller} (See also \cite{NicolasPy}).

\item Infinite groups with property (QT) admit an action on a finite product of quasitrees whose orbit maps yield a quasi-isometric embedding, which are in particular proper \cite{BBFQt}. In this case, due to properness, the action is AU-acylindrical. 

\item The class of colorable Hierarchically Hyperbolic Groups (abbreviated as HHGs), or more generally any finitely generated group that admits an action $\G\to \Aut\X$ for which the orbit maps are quasi-isometric embeddings \cite{HagenPetyt}. Again, due to properness, the action is AU-acylindrical.

\item Groups acting properly on products of trees, such as Coxeter groups and RAAGs \cite{Januszkiewicz, Hosaka, Button}, or finitely generated subgroups of $\PGL_2K$, where $K$ is a global field of positive characteristic \cite{FisherLarsenSpatzierStover}.

\item Finitely generated infinite groups that are quasi-isometric to a finite product of proper cocompact non-elementary $\delta$-hyperbolic spaces in fact act geometrically, and hence acylindrically on a (possibly different) finite product of rank-one symmetric spaces and locally finite $\delta$-hyperbolic graphs \cite[Theorem K]{Margolis}. 

\item\label{Ex:S-arith}  Finitely generated subgroups $\G < \SL_2 \~\Q$, where $\~\Q$ is the algebraic closure of $\Q$. See Section~\ref{sec: SL2} for more details.
 
\item Limit groups (in the sense of Sela) or fully residually free (in the sense of Kharlampovich and Miasnikov) are relatively hyperbolic with respect to their higher-rank maximal abelian subgroups \cite{Alibegovic,Dahmani}. Subdirect products of such have been studied \cite{BridsonHowieMillerShort}, \cite{KochloukovaLopezdeGamizZearra}. (See also \cite{BaumslagRoseblade}.)
\end{enumerate}

\begin{remark}
    Mapping class groups importantly appear in both (1) and (4) (with $D>1$), each of which confers different information and are therefore both of interest. We take the opportunity to highlight the importance of taking a ``representation theoretic" point of view, namely to value all such actions.  
\end{remark}

\subsection{The Iwasawa Decomposition and  Godement's Compactness Criterion}\label{sec: Godement} 

As has been discussed above, AU-acylindricity generalizes the notion of a proper action, which in turn generalizes the type of action a lattice enjoys on its ambient group. In the context of linear groups, unipotent elements play an important role, as the reader will see throughout the rest of this section. 

If $G$ is a semisimple  Lie group (over a local field) then it admits an \emph{Iwasawa decomposition} $G=KA\mathcal N$, where $K\leq G$ is a maximal compact subgroup, $A\mathcal N\leq G$ is a maximal solvable subgroup, $\mathcal N\leq A\mathcal N$ its maximal unipotent subgroup, and $A$ the maximal diagonalizable subgroup \cite[Theorem 29.3]{Bump}, \cite[Proposition 11.100]{AbramenkoBrown}. (Having fixed a linear representation, a unipotent subgroup is such that all eigenvalues are equal to 1, and diagonalizable means that the representation can be conjugated in an algebraically closed field so as to make the subgroup diagonal.) The symmetric space or Bruhat-Tits building associated to $G$ is $Z=G/K$ is a CAT(0) space (with the $\ell^2$-product metric if $G$ is not strongly irreducible) and has Furstenberg boundary  $G/A\mathcal N$. 

If $G$ is a simple Lie group over $\R$ of rank-one  then $Z$ is one of the algebraic hyperbolic spaces \cite{DasSimmonsUrbanski} and the visual boundary of $Z$ as a CAT(0) space can be naturally identified with the Furstenberg boundary, i.e.   $\partial Z\cong H/\mathcal N$ \cite[IV.4.50.Lemma]{GuivarchJiTaylor}, \cite[Appendix 5.6]{BallmannGromovSchroeder}.

Moreover, the Iwasawa decomposition yields the following connections between algebra and geoemtry:

\begin{itemize}
    \item an element $\g\in G$, respectively a subgroup $H$ is elliptic if and only if it can be conjugated into $K$;
      \item an element $\g\in G$, respectively a subgroup $H$ is parabolic if and only if it can be conjugated into $\mathcal N$ \cite[Appendix 5.4]{BallmannGromovSchroeder};
    \item an element  $\g\in G$ is loxodromic, respectively a subgroup $H$ is lineal if and only if it can be conjugated into $A$;
    \item  a subgroup $H$ is quasi-parabolic if and only if it can be conjugated into $A\mathcal N$ and is not lineal or parabolic;
     \item  a subgroup $H$ is general type (with full limit set) if and only if it is Zariski-dense.
\end{itemize}

Unipotent elements also play a staring role in Godement Compactness criterion. This conjecture was proved for standard lattices (i.e. $G(\Z)\leq G(\R)$) by Borel and Harish-Chandra \cite[Theorem 3]{BHC} and later generalized by Behr to the $S$-arithmetic case \cite[Theorem C]{Behr} (see also \cite[Theorem 5.8]{Benoist}). 

\begin{theorem}[Godement Compactness criterion] \label{thm: Godement}
Let $G(\Q)\leq \SL_n\Q$ be an algebraic group defined over $\Q$, and $S\subset \Z$ be a finite set of primes. The lattice $G(\Z[S^{-1}])\leq G(\R)\times\Prod{p\in S}{}G(\Q_p)$ is cocompact if and only if $G(\Z)\leq G(\R)$ is cocompact if and only if $G(\Z)$ has no nontrivial unipotent elements.
\end{theorem}

Therefore, combining these results with Theorem \ref{thm:elimmixedfactors} we obtain:

\begin{cor}\label{Cor: cocompact iff acyl}
    Let $G(\Q)\leq \SL_n\Q$ be an algebraic group defined over $\Q$, and $S\subset \Z$ be a finite set of primes such that the irreducible factors of $G(\R)\times\Prod{p\in S}{}G(\Q_p)$ have rank one. The lattice $G(\Z[S^{-1}])\leq G(\R)\times\Prod{p\in S}{}G(\Q_p)$ acts acylindrically on $\X$ the product of algberaic hyperbolic spaces  and Bruhat-Tits trees if and only if $G(\Z)\leq G(\R)$ is cocompact if and only if $G(\Z)$ has no nontrivial unipotent elements.
\end{cor}

Indeed, a unipotent element in $G(\Z)$ will project to a unipotent element in each irreducible factor of $G(\R)$. Since these factors are  assumed to be rank-one, such an element from $G(\Z)$ is either the identity if $G(\Z)$ is cocompact (and hence the action is acylindrical) or not the identity (and hence the action is nonuniformly acylindrical).

\begin{remark}
    This allows us to contextualize the obstructions found in Section \ref{sec: obstructions/simplif}. As we saw above, if $G(\R)$ has rank-one factors, then unipotent elements in $G(\Z)$ will correspond to elements that act as parabolic isometries in each factor associated to $G(\R)$. 
\end{remark}

The reader will see that unipotent elements continue to reappear below. We also point the reader to \cite[2.4.Proposition]{LubotzkyMozesRaghunathan} which states that distortion characterizes unipotent elements.

\subsection{Semi-Simple Lattices and Superrigidity}

In what follows, an action of $\G$ on a hyperbolic space $X$ is called \emph{coarsely minimal} if $X$ is unbounded, the limit set of $\G$ in $\partial X$ is not a single point, and every quasi-convex $\G$-invariant subset of $X$ is coarsely dense. Given a metric space $X$ and $C \geq 0$, denote by $Bdd_C (X)$ the set of all closed
subsets of diameter at most $C$, endowed with the Hausdorﬀ metric. Notice that
$Bdd_C (X)$ is quasi-isometric to $X$. Moreover, two actions of a group $\G$ on metric spaces $X_1,X_2$ are said to be \emph{equivalent} if there exists a $\G$-equivariant quasi-isometry $X_1 \to Bdd_C (X_2)$ for some $C \geq 0$.  

\begin{remark}\label{rem: equiv and AU}
   Let $X_1, \dots, X_\D,X_1', \dots, X_\D'$ be $\delta$-hyperbolic spaces, and set $\X= \Prod{i=1}{
    \D} X_i$, and $\X'= \Prod{i=1}{
    \D} X_i'$. It is straightforward to check that if $\G$ acting on $X_i$ is equivalent to $\G$ acting on $X'_i$ for all $1 \leq i \leq \D$, then the action $\G\to \Aut_{\! 0}\X$ is (AU-)acylindrical if and only if the action $\G\to \Aut_{\! 0}\X'$ is (AU-)acylindrical. 
\end{remark}

 We note that, the following theorem applies to reducible lattices as well, provided that the hypotheses are met for the corresponding factors, which is how it is originally stated in \cite{BaderCapraceFurmanSisto}. We also note that if $\D=0$ then no coarsely minimal actions exist, which was first proved by Haettel \cite{Haettel} (see Theorem~\ref{Thm: Haettel HR} below). 

\begin{theorem}\cite[Theorem 1.1]{BaderCapraceFurmanSisto} \label{Thm: BCFS} Fix integers ${\mathbf{D}}\geq 0$ and $F\geq 0$. Consider an irreducible lattice $\G\leq G$  where $G= \Prod{i=1}{{\mathbf{D}}}H_i\times \Prod{j=1}{F}G_j$ is a product of centerless simple Lie groups. Assume further that $H_1, \dots, H_{\mathbf{D}}$ have real rank one, and $G_1, \dots, G_F$ have real rank strictly larger than one. 
If ${\mathbf{D}}\geq 2$ or $F>0$ then any coarsely minimal isometric action of $\G$ on a geodesic hyperbolic space is equivalent to one of the actions
$$\G \longrightarrow G \stackrel{pr_i}{\longrightarrow} H_i \stackrel{\rho_i}{\longrightarrow} \Isom S_i \hspace{10pt} (1 \leq i\leq {\mathbf{D}})$$
where $S_i$ is the hyperbolic symmetric space of the factor $H_i$, $i=1, \dots, {\mathbf{D}}$.
\end{theorem}

\begin{remark}\label{rem: finitely many inner}
Fix $i\in \{1,\dots, {\mathbf{D}}\}$ and set $H=H_i$, $S=S_i$, $\rho=\rho_i$. Theorem~\ref{Thm: BCFS} does not address the quality of the homomorphism $\rho: H\to \Isom S$. The proof relies on coarse equivalence, which in general does not produce a measurable map. (Note that measurable homomorphisms between Lie groups are continuous \cite{Kleppner} and  hence smooth \cite[Corollary 3.50]{HallLieGrps}). We shall argue that $\rho$  can be taken to be rational automorphism of $H$, achievable by conjugation in $\Isom S$. First, the coarse minimality condition ensures that the image is not trivial. Since $\G$ is a higher-rank irreducible lattice, we may apply Margulis Superrigidity \cite{Margulis} to deduce that $\rho|_{\G}$ extends to a rational map $\rho': H\to \Isom S$, and so we may replace $\rho$ with $\rho'$ if necessary. (Note that there are uncountably many nonmeasurable homomorphisms  $SL_2\R\to \SL_2\R$ that restrict to the identity on $\SL_2\Z$.)

Since $H$ is connected and $S$ is its symmetric space, we have that $H$ is the connected component of the identity and has finite index in $\Isom S$. Let  $\Aut^{\! \8} H$ denote the group of smooth automorphisms of $H$. Since $H$ is rank one, we have that $\Inn H \leq \Aut^{\! \8} H$ is of index at most 2, precisely corresponding to the index $H\leq \Isom S$  (see \cite[Corollary 2.15]{Gundogan}, also \cite{ChuahZhang}). Therefore, any smooth map $H\to \Isom S$ is in the same $\Isom S$-conjugacy class of the identity embedding $H\hookrightarrow \Isom S$. We shall call the corresponding action of $\a:\G\to \Isom S$ the \emph{standard action of $\G$ on $S$}. This is touched on in \cite[Example 1.4]
{BaderCapraceFurmanSisto}. 
\hfill $\blacksquare$
\end{remark}

In Part II \cite{BFPart2} (see also \cite[Lemma 4.26]{BF}) we prove the following result. (This is alluded to in \cite[Remark 4.3]{BaderCapraceFurmanSisto}; a similar idea is also considered in \cite[Lemma 2.21]{PropNL}.) The term ``essential" is inspired by a similar construction in \cite{CapraceSageev}. We shall use it to prove the results that follow. 

\begin{prop}[The Essential Core]\label{prop: ess core}
    Let $\rho:\G\to \Isom X$ be of general type. There exists a $\rho(\G)$-invariant quasiconvex $\delta'$-hyperbolic subspace $\mathcal E_\rho(X)\subset X$ on which  the restricted action $\rho(\G)|_{_{\mathcal E_\rho(X)}}$ is coarsely minimal. In particular, any  $\rho(\G)|_{_{\mathcal E_\rho(X)}}$-invariant quasiconvex subset of $\mathcal E_\rho(X)$ is coarsely dense in $\mathcal E_\rho(X)$ and $\partial \mathcal E_\rho(X) = \Lambda(\rho(\G))$ is infinite.
\end{prop}

\begin{theorem}\label{thm: HR acyl iff cocompact}
    Let $\G\leq G$ be an irreducible lattice as in Theorem~\ref{Thm: BCFS} and $S_i$ the hyperbolic symmetric space for the factors  $H_i$, $i=1, \dots, {\mathbf{D}}$. There exists an AU-acylindrical action $\G\to \Aut\X$ with general type factors if and only if $F=0$, and up to passing to the essential cores and equivalence of the factors of $\X$, $\X =\mathbb S\times \mathbb S_{\mathrm{dup}}$, where $\mathbb S= \Prod{i=1}{\mathbf{D}}S_i$ and $\mathbb S_{\mathrm{dup}}$ is the (possibly empty) product of duplicate copies of factors of $\mathbb S$. Furthermore, the action is acylindrical if and only if $\G\leq G$ is cocompact.
\end{theorem}

\begin{proof} 
   Suppose $F=0$. Then $G= \Prod{i=1}{{\mathbf{D}}} H_i$ is a product of centerless simple Lie groups of real rank one and $\G$ acts properly (and with finite covolume) on $\mathbb S$ the corresponding product of hyperbolic symmetric spaces. Hence, the standard action $\Delta(\a_1, \dots, \a_{\mathbf{D}}): \G\to \Aut\mathbb S$  is AU-acylindrical. Furthermore, by the Tits Alternative for linear groups, the factor actions either have a finite orbit on the boundary or are of general type. Since the projection of $\G\to H_i$ is Zariski-dense, the action is of general type for each $i=1, \dots, {\mathbf{D}}$. Furthermore, if $\G$ is cocompact, then the action is uniformly proper and hence acylindrical. 

   Conversely, suppose that $\rho: \G\to \Aut\X$ is AU-acylindrical with general type factors, where $\X=\Prod{i=1}{\D}X_i$.  Since a finite index subgroup of $\G$ is still an irreducible lattice in $G$ and AU-acylindricity with general type factors is inherited by finite index subgroups, without loss of generality, we assume that  $\rho(\G)\leq\Aut_{\! 0}\X$.
   
   Since the factors are of general type, by Proposition~\ref{prop: ess core} they each have $\G$-invariant subspaces on which the action is coarsely minimal. Since the subspaces are invariant, the action on the product of the subspaces is still an AU-acylindrical action on a finite product of hyperbolic spaces with general type factors. So without loss of generality, we may assume that the factor actions are in fact coarsely minimal.
   
 By Theorem~\ref{Thm: BCFS} and Remark~\ref{rem: finitely many inner}, up to equivalence, for each $i\in \{1, \dots, \D\}$ there is a $j(i) \in \{1, \dots, {\mathbf{D}}\}$ such that the factor action $\rho_i: \G \to \Isom X_i$ is equivalent to the standard action  $\a_{j(i)}:\G\to \Isom S_{j(i)}$. Combining these factors, we get  an equivalence between the action $\rho:\G \to \Aut_{\! 0}\X$ and the diagonal of these standard actions  $\Delta\Big(\a_{j(1)}, \dots, \a_{j(\D)}\Big):\G\to \Aut_{\!0}\Big( \Prod{i=1}{\D}S_{j(i)}\Big)$. As discussed in Remark~\ref{rem: equiv and AU}, that action on $\Prod{i=1}{\D}S_{j(i)}$ is AU-acylindrical. 

 By Lemma~\ref{lem: duplicate action}, any duplicate factors in $\Prod{i=1}{\D}S_{j(i)}$ can be removed. Therefore, we may assume that $\D\leq {\mathbf{D}}$ and up to relabeling,  $S_{j(i)}= S_i$.  
 Since $\Prod{i=1}{\D}S_i$ is locally compact, AU-acylindricity is equivalent to properness of the action. However, since $\G$ is irreducible, if $F>0$ or $I\subsetneq \{1, \dots, {\mathbf{D}}\}$ the projection of $\G$ to $\Prod{i\in I}{}H_i$ is dense and therefore, the corresponding action $\G\to \Prod{i\in I}{}\Isom S_i$ is not proper. 
 This proves that $F=0$ and $\D={\mathbf{D}}$.

Finally, if the action is acylindrical then $\G$ is in fact cocompact. Indeed, by the Godement Compactness Criterion (Theorem~\ref{thm: Godement}) $\G$ is cocompact if and only if it has no non-trivial unipotents. Unipotents are precisely those elements that will act parabolically in each factor (see Section~\ref{sec: Godement}). Therefore,  by Theorem~\ref{thm:elimmixedfactors}, no such elements in $\G$ can exist, i.e. $\G$ has no nontrivial unipotents.
\end{proof}

 \subsection{The Special Linear Group of Dimension 2}\label{sec: SL2}

 Recall that if $\mathfrak R\subset \mathbb C$ is a ring then $SL_2\mathfrak R$ is the set of matrices with determinant 1 with entries from $\mathfrak R$. 
 
 We now consider a class of groups that is not addressed in Theorem~\ref{Thm: BCFS}  (nor in Theorem~\ref{Thm: Haettel HR} below) -- $\SL_2\mathfrak R$, where $\mathfrak R=\O_K[S^{-1}]$ is the ring of $S$-integers in $K$ a finite extension of $\Q$. This class (and its subgroups) was mentioned in Example (\ref{Ex:S-arith}) above.

\begin{defn} 
    Let $\mathfrak R$ be a ring. The group $\mathrm{EL}_2\mathfrak R$ is the group generated by \emph{elementary} matrices, i.e. $2\times 2$ matrices $g$ such that there exists  $r\in \mathfrak R$ for which $$g= \begin{pmatrix}
        1 & r \\ 
        0 & 1
    \end{pmatrix}  \text{ or } \begin{pmatrix}
        1 & 0 \\
        r & 1
    \end{pmatrix}.$$
\end{defn}

\begin{defn}
    Let $U_1$, $U_2$ be the upper and lower unipotent subgroups of $\mathrm{EL}_2\mathfrak R$ respectively. The group $\mathrm{EL}_2\mathfrak R$ is \emph{elementary boundedly generated} if there exists $n_0>0$ so that for every $g\in \mathrm{EL}_2\mathfrak R$ there exists $g_1, \dots, g_{n_0}\in U_1\cup U_2$ such that $g= g_1\cdots g_{n_0}$. 
\end{defn}

We will need the following result about finitely generated rings of algebraic numbers, their group of units, and the relationship to bounded generation and acylindrical hyperbolicity. The first item follows immediately by observing that the lower (respectively upper)  triangular group in $\SL_2 \mathfrak R$ is isomorphic to  $\mathfrak R^\times \ltimes \mathfrak R$, where the action $\mathfrak R^\times \to \Aut \mathfrak R$ is given by multiplication by the square of an element (respectively the square of its inverse). 

\begin{theorem}\label{thm: fin gen ring and units}
Let $\mathfrak R\subset \~\Q$ be a finitely generated ring of algebraic numbers.
\begin{enumerate}
    \item\label{units and uper triang} The group of units $\mathfrak R^\times$ is finite if and only if the \emph{upper} triangular group in $\SL_2 \mathfrak R$ is virtually abelian if and only if the \emph{lower} triangular group in $\SL_2 \mathfrak R$ is virtually abelian.
    \item\cite{Vaserstein}\label{Vaserstein}  If $\O_K[S^{-1}]$ has infinitely many units then $\SL_2\O_K[S^{-1}]= \mathrm{EL}_2\O_K[S^{-1}]$ 
   \item\cite{MorganRapinchukSury}\label{MRS} The ring $\O_K[S^{-1}]$ has infinitely many
units if and only if the group $\SL_2 \O_K[S^{-1}]$ is elementary boundedly generated, and in this case every element can be expressed as a product of at most 9 elementary matrices. 
\item\label{OO1} There exists a finite field extension $K$ of $\Q$, and a finite set of valuations $S$ such that $\mathfrak R\leq \O_K[S^{-1}]$ are finite index inclusions. Furthermore, $\mathfrak R$ has infinitely many units if and only if $\O_K[S^{-1}]$ has infinitely many units if and only if $S\neq\varnothing$ or $K\not\subset\Q(\sqrt{-a})$, where $a\in \N$.

\end{enumerate}
\end{theorem} 

The last unjustified item is currently folklore, but its proof will appear in a forthcoming update of \cite{OsinOyakawa}.
 We can now prove the following result about $\SL_2\mathfrak R$; the proof follows the strategy developed in \cite{ShalomBddGen(T)}.
Recall that actions on bounded metric spaces are always acylindrical.

\begin{prop}\label{prop: loc comp nonamen proper}
    For each $i=1, \dots, \D$, let $X_i$ be unbounded locally compact $\delta$-hyperbolic spaces such that if $\xi\in \~X_i$ then $\fix\{\xi\}\leq {\Isom X_i}$ is amenable. Let $\G$ be nonamenable and $\rho_i: \G\to \Isom X_i$ have amenable kernel such that $\Delta(\rho_1, \dots, \rho_\D): \G\to \Aut_{\! 0}\X$ is proper. If $I \subset \{1, \dots, \D\}$ is the subset of indices for which $\rho_i(\G)$ is general type, then $I\neq \varnothing$ and $\Delta(\rho_i: i\in I):\G\to \Aut_{\! 0}\Prod{i\in I}{}X_i$ is proper.
\end{prop}

\begin{proof}
    Since $\G$ is nonamenable and the kernels of $\rho_1, \dots, \rho_\D$ are amenable, the images $\rho_i(\G)$ must be nonamenable. Furthermore,  the hypotheses on amenability of boundary point stabilizers and the classification of actions Theorem~\ref{thm:actionsclassrk1} eliminates all possibilities except elliptic and general type. 

    By Lemma~\ref{Lem:AU-acyl has finite kernel} the kernel $\Delta(\rho_1, \dots, \rho_\D)$ is finite and since $\G$ is nonamenable, it is infinite. Therefore, not all factor actions are elliptic, in particular, the set of indices $I$ for which $\rho_i(\G)$ is general type is not empty. By Lemma~\ref{Lem:elimellfactors} we have that $\Delta(\rho_i: i\in I):\G\to \Aut_{\! 0}\Prod{i\in I}{}X_i$ is proper.
 \end{proof}

\begin{remark}\label{rem:amen elem}
    Let $X$ be one of the standard algebraic hyperbolic spaces or an infinite locally finite tree. As discussed in Section~\ref{sec: Godement}, or proved in \cite[Section 4]{PaysValette}, $\fix\{\xi\}$ is amenable for every $\xi\in \partial X$, and therefore, these spaces satisfy the hypotheses of Proposition~\ref{prop: loc comp nonamen proper}.
\end{remark}

\begin{cor}\label{cor:fin gen}
    If $\G\leq \SL_2\~\Q$ be finitely generated and nonamenable then there exists a proper action with general type factors $\G\to \Aut_{\! 0}\X$, where $\X$ is a product of finitely many copies of the hyperbolic plane, hyperbolic 3-space and Bruhat-Tits trees.
\end{cor}

\begin{proof} Since $\G$ is finitely generated, there exists a finite field extension $\Q\subset K$, and $S$ a finite set of valuations on $\Q$ such that $\G\leq \SL_2\O_K[S^{-1}]$. The latter is a lattice in $\Prod{i=1}{r} \SL_2\R \times \Prod{j =1}{c} \SL_2\mathbb{C} \times \Prod{s \in S}{}\SL_2 K_s$ which acts properly (and with finite covolume) on the associated product of rank-one symmetric spaces and Bruhat-Tits trees, and hence the action of $\G$ is proper as well. Furthermore, the factor actions are injective and hence have trivial kernel.
By Remark~\ref{rem:amen elem}, these factors have amenable boundary-point stabilizers and the result follows from Proposition~\ref{prop: loc comp nonamen proper}.

\end{proof}

\begin{prop}\label{prop:SL2ZS} 
Let $\mathfrak R\subset \~\Q$ be a finitely generated infinite ring of algebraic numbers and let $\G=\SL_2 \mathfrak R$. Then there exists a nonelliptic acylindrical action $\G\to \Aut\X$ if and only if  $\mathfrak R^\times$ is finite, in which case $\G$ is acylindrically hyperbolic.

\end{prop}

\begin{proof}
Assume that $\mathfrak R^\times$ is infinite and $\rho:\G\to \Aut \X$ is acylindrical. We will deduce that the action is elliptic. By Theorem~\ref{thm: fin gen ring and units}.\ref{OO1}, there exists a finite field extension $K\supset \Q$ and a finite set of non-archimdean valuations $S$ on $K$ such that $\mathfrak  R\leq \O_K[S^{-1}]$. Therefore,  $\G=\SL_2 \mathfrak R\leq \SL_2 \O_K[S^{-1}]$ are finite index inclusions of groups. Moreover, since $\mathfrak R$ has infinitely many units, so does $\O_K[S^{-1}]$. Let $\G':=\SL_2 \O_K[S^{-1}]$ and $F$ denote the set of cosets $\G'/\G$.

By Proposition~\ref{prop: fin ind induction}, setting $\X'=\Prod{\chi\in F}{}\X$ there exists  $\rho':\G'\to \Aut \X'=\Sym(F)\ltimes \Aut \X'$ that is acylindrical and such that $\rho'(\G)$ does not permute the factor corresponding to $\chi_0$, the identity coset in $F$. 
Since projections are 1-Lipschitz, in particular, if $(\mathbf{x}_\chi)_{\chi\in F}, (\mathbf{y}_\chi)_{\chi\in F}\in \X'$ then 
$$d(\mathbf{x}_{\chi_0}, \mathbf{y}_{\chi_0})\leq d\big((\mathbf{x}_\chi)_{\chi\in F}, (\mathbf{y}_\chi)_{\chi\in F}\big).$$

Therefore, to deduce that the action of $\rho(\G)$ on $\X$ is elliptic, it is sufficient to deduce that the $\rho'(\G)$ action on $\X'$ is elliptic. To this end, let  $U',L'\leq \G'$ be the upper and lower  triangular groups. Since the group of units in  $\mathfrak R$ is infinite, so is that in $\O_K[S^{-1}]$ and hence $U'$ and $L'$ are not virtually abelian.
  By the Tits Alternative Theorem~\ref{intro:titsalt}, it follows that the action of $\rho'(U')$ and $\rho'(L')$ on $\X'$ must be elliptic.
  
  We now deduce that the action $\rho'(\G')$ is also elliptic.
  Fix $(\mathbf{x}_\chi)_{\chi\in F}\in \X'$ and $R>0$ such that $d((\mathbf{x}_\chi)_{\chi\in F},\rho'(g)(\mathbf{x}_\chi)_{\chi\in F})\leq R$ for all $g\in U'\cup L'$. By Theorem~\ref{thm: fin gen ring and units}.\ref{MRS} $\G'$ is boundedly generated by $U'$ and $L'$, i.e. for every $g'\in \G'$ there exists $g_1, \dots, g_9\in U'\cup L'$ such that $g'= g_1 \cdots g_9$. We then have that 

  \begin{eqnarray*}
      d\Big((\mathbf{x}_\chi)_{\chi\in F}\bbcomma\rho'(g)(\mathbf{x}_\chi)_{\chi\in F}\Big)&=& d\Big((\mathbf{x}_\chi)_{\chi\in F}\bbcomma\rho'(g_1 \cdots g_9)(\mathbf{x}_\chi)_{\chi\in F}\Big)\\
  &\leq & \Sum{i=1}{9}d\Big((\mathbf{x}_\chi)_{\chi\in F}\bbcomma\rho'(g_i)(\mathbf{x}_\chi)_{\chi\in F}\Big)\leq 9R.
  \end{eqnarray*}
  In particular, $\rho'(\G')$  and hence the subgroup $\rho'(\G)$ are acting elliptically.

Conversely, let $\mathfrak R^\times$ be finite. Then $\mathfrak{R}\leq \mathcal{O}_K$ is a finite index inclusion of abelian groups, where $\mathcal{O}_K$ is the ring of integers in $K= \Q(\sqrt{-a})$, for some $a\in\N$. Therefore, $\SL_2\mathfrak{R}\leq \SL_2\mathcal{O}_K$ is also finite index. In this case, $\SL_2\mathcal{O}_K$ is a nonuniform lattice in $\SL_2 \R$ (respectively $\SL_2\mathbb C$), according to whether $a=0$ (respectively $a>0$). In either case, the standard action on the associated hyperbolic space is general type and proper. Furthermore, any matrix whose trace has absolute value larger than 2 is a WPD element. Therefore,  $\SL_2\mathcal{O}_K$ is acylindrically hyperbolic  and so is $\G$. \end{proof}

\subsection{Genericity for Unipotent-Free}
In this section, we deduce that admitting an AU-acylindrical action on a CAT(0) space is generic among the finitely generated subgroups of $SU_n$. The key to this is the following result by Douba.

\begin{theorem}\cite[Theorem 1.1.i]{Douba}
Let $\G\leq \SL_n\mathbb C$ be finitely generated such that no non-trivial element of $\G$ is unipotent. Then $\G$ acts properly on a CAT(0) space. 
\end{theorem}

If $n=2$ then the associated CAT(0) space will be the $\ell^2$-product of copies of $\mathbb{H}^2, \mathbb{H}^3$ and (not necessarily locally finite) Bruhat-Tits trees, as above.

The reader may wonder how to find examples of such groups that are free of non-trivial unipotents: Consider the special unitary group $\mathrm{SU}_n\leq \SL_n\mathbb C$, which is the group of matrices whose inverse is given by the transpose-complex conjugate, i.e.  $AA^*=I$. By way of this equation, we see that all eigen-values of a unitary matrix must belong to the unit circle in $\mathbb C$, and in particular, the only unipotent element (i.e. the only element to have sole eigen value 1) in $\mathrm{SU}_n$ is the identity. Moreover, if $n>1$ then $\mathrm{SU}_n$ is not virtually solvable, and so by the Tits Alternative, it contains nonamenable groups.  Using Lemma~\ref{Lem:acyl+loc comp implies unif proper_p1}, we obtain the following corollary. Note that $\mathrm{SU}_n$ can also be replaced with any group within its isogeny class; isogenies are algebraic isomorphisms up to finite kernel.

\begin{cor}
    Let $n_1, \dots, n_k\in \N$, and $\G\leq \Prod{i=1}{k}\mathrm{SU}_{n_i}$ be finitely generated. Then $\G$ acts properly (and hence AU-acylindrically) on a CAT(0) space. If in fact $\G$ can be conjugated in to $\mathrm{SU}_{n_i} \cap \SL_{n_i}(\~\Q)$ for each $i=1, \dots, k$ then the CAT(0) space is uniformly locally compact and hence the $\G$ action is acylindrical.  If $\{n_1, \dots, n_k\}=\{2\}$, then the CAT(0) space is a finite product of copies of $\mathbb H^2$, $\mathbb H^3$, and locally finite trees.  
\end{cor}

\subsection{Groups with Property (NL)}\label{Subsec: NL}

In \cite{PropNL}, the authors systemize the study of groups that do not admit \emph{any} actions on hyperbolic spaces with a loxodromic element -- such groups are said to have Property $\nl$.  In other words, the only  actions such a group can admit on a hyperbolic space are elliptic or parabolic. If every finite index subgroup also has Property $\nl$, the group is said to be \emph{hereditary} $\nl$. 

A primary example of a group with this property comes from higher-rank lattices. Note that if $F=1$ in the following statement, the irreducibility condition should be ignored. We also note that the lattices in the below theorem do act properly on the associated product of model geometries. 

\begin{theorem}\label{Thm: Haettel HR}\cite[Theorem A]{Haettel} Let $F\geq 1$ and $\G\leq G$ be an irreducible lattice in $G=\Prod{i=1}{F}G_i$, where $G_i$ is a higher rank, almost simple connected, algebraic group with finite center over a local field, for $i=1, \dots, F$. Then any action of $\G$ by isometries on a $\delta$-hyperbolic space is  either elliptic or parabolic.
\end{theorem}

Further examples of groups with Property $\nl$, besides the obvious finite groups and the higher-rank lattices considered above,  include Burnside groups, Tarskii monster groups, Grigorchuk groups, Thompsons groups $T,V$ and many ``Thompson-like" groups  (see \cite[Theorem 1.6]{PropNL} for a more extensive list). It is worth noting however that Thompson's group $V$ and similar-type diagram groups act properly on infinite-dimensional CAT(0) cube complexes \cite{Farley}. Furthermore, while finitely generated torsion groups do not act on 2-dimensional CAT(0) complexes \cite{NorinOsajdaPrzytycki},  free Burnside groups do act non-trivially on infinite dimensional CAT(0) cube complexes \cite{Osajda}. Finitely generated bounded-torsion groups are conjectured not to admit fixed-point free actions on locally compact CAT(0) spaces. In a personal communication, Coulon and Guirardel have told us of some upcoming work: there is a finitely generated infinite amenable (unbounded) torsion group G which admits a proper action on an infinite dimensional CAT(0) cube complex.

Returning to our context, it follows from \cite[Corollary 6.5]{PropNL}, that groups with the hereditary $\nl$ property also do not admit interesting actions on products of hyperbolic graphs. In particular, it follows from the proof of \cite[Proposition 6.4]{PropNL} and Theorem~\ref{thm:elimmixedfactors}, that if $\G$ is a hereditary $\nl$ group, and $\X$ is a product of hyperbolic spaces, then $\G \to \Aut \X$ is acylindrical only if the action is elliptic in each factor.


\section{Elementary subgroups and the Tits Alternative}\label{sec:elemsubgrp}

In this section, we explore the structure of the stabilizer of a pair of points  in the regular boundary $\partial_{reg}\X= \Prod{i=1}{\D}\partial X_i$. If the points have distinct factors and the action is AU-acylindrical, we will prove that the stabilizer is virtually free abelian of rank no larger than $\D$. This is an analog of the result in rank-one \cite[Proposition 6]{BestvinaFujiwara}, \cite[Theorem 1.1]{Acylhyp}, and for HHGs \cite[Theorem 9.15]{DurhamHagenSisto}, although our proof uses new techniques. We also consider the stabilizer of a single point in the regular boundary in a group acting acylindrically and obtain the same result, which is as expected according to our philosophy (see Section~\ref{sec: Godement}).

\begin{defn}\label{def: elem subgroup}
Given $\G \to  Aut\X$ and $\xi^-, \xi^+\in \partial_{reg}\X$ with distinct factors, i.e. $\xi_i^-\neq \xi_i^+$ for each $i=1, \dots, \D$, the subgroup  $\Cap{i=1}{\D}\fix_\G\{\xi_i^-, \xi_i^+\}$ is the associated \emph{elementary subgroup}, denoted by $E_\G(\xi^-, \xi^+)$. Recall that $\fix_\G\{\xi_i^-, \xi_i^+\}$ is the set of elements that fix the set $\{\xi_i^-, \xi_i^+\}$ pointwise. If $\g\in \G$ is a regular element, then the associated \emph{elementary subgroup} is $E_\G(\g):= E_\G(\g^-, \g^+)$, where $\g^-, \g^+\in \partial_{reg}\X$ are respectively the repelling and attracting fixed points for $\g$. 
\end{defn}

Our main goal is to prove the following results. 

\begin{prop}\label{prop:elemsubvirtab} If $\G\to \Aut \X$ is AU-acylindrical and not elliptic, and $\xi^-, \xi^+ \in \partial_{reg}\X$ with distinct factors then $E_\G(\xi^-, \xi^+)$ is virtually  $\Z^k$ where $1\leq k \leq\D$. \end{prop}

\begin{prop}\label{prop: reg fixed implies reg pair} Assume $\G \to \Aut\X$ is acylindrical and not elliptic. If $\xi \in \partial_{reg}  \X$ then either $\stab \{\xi\}$ is finite or there exists an $\xi' \in \partial_{reg} \X$ such that $\xi \neq \xi'$ and $\fix_\G\{\xi\} = \fix_\G\{\xi, \xi'\}$. Furthermore, if all factor actions are neither elliptic nor parabolic, then $\xi_i'\neq \xi_i$ for each $i=1, \dots, \D$.
\end{prop} 

Our strategy to prove the above propositions is as follows:

\begin{enumerate}
    \item We first prove, under the hypotheses of Proposition \ref{prop:elemsubvirtab} that $E_\G(\xi^-, \xi^+)$ is amenable by showing that all finitely generated subgroups have polynomial growth  (of degree at most $\D$). 
    \item Using amenability, the Busemann quasimorphisms are now  homomorphisms and  the associated product of these allows the action to descend to $\R^{\D}$. We prove that the AU-acylindricity of the action also descends. An application of Lemma~\ref{lem: AU eucl is Zk} then proves Proposition~\ref{prop:elemsubvirtab}.
    \item Without loss of generality, up to passing to a finite index subgroup that preserves the conditions, we may assume that $\G \to \Aut_{\!0} \X$. We then deduce Proposition~\ref{prop: reg fixed implies reg pair}  by utilizing the uniformity  of acylindricity to exclude the possibility of a quasi-parabolic factor and hence  find another  point in $\partial_{reg} \X$ which is also fixed. We then conclude by applying Proposition~\ref{prop:elemsubvirtab} and noting that the structure of $\fix_\G\{ \xi \}$ (whether finite or infinite) lifts back to $\Aut \X$. 
\end{enumerate}

Assuming the above results hold, we can immediately prove our version of the Tits Alternative as stated in Theorem~\ref{intro:titsalt}. 

\begin{proof}[Proof of Theorem~\ref{intro:titsalt}] Without loss of generality, we may first pass to a finite index subgroup if necessary and assume that $\G$ is factor-preserving, i.e. assume $\G\to \Aut_{\! 0}\X= \Prod{i=1}{\D}\Isom X_i$. Indeed, the finite index subgroup must also act non-elliptically and the acylindricity is naturally inherited. Since the action is not elliptic, by Corollary~\ref{cor: clean up acyl}, we may assume that all factors are neither elliptic nor parabolic. 

If any projection $\G\to \Isom X_i$ is general type then $\G$ contains a free group by the classification of actions Theorem~\ref{thm:actionsclassrk1}.  Therefore, we may assume that no factor is of general type.

We are then in the situation that each factor is lineal or quasi-parabolic. Up to passing to a finite index subgroup once more, we may assume that each factor is either oriented lineal or quasi-parabolic. This means the action is not elliptic and fixes some point $\xi\in \partial_{reg}\X$. 
By applying Proposition~\ref{prop: reg fixed implies reg pair} first and then applying Proposition~\ref{prop:elemsubvirtab}, the result follows.
\end{proof}

We now establish some preliminary results that we will need to prove the propositions listed above. We start with an examination of actions on quasilines in rank-one and products of such in higher rank in the next two subsections. 

\subsection{Rank-one}\label{Sec:elem rank1}

Fix  distinct boundary points $\xi^-, \xi^+\in \partial X$. Let $L$ be the union of all $(1, 20\delta)$ quasi-geodesics that limit to $\xi^-$ and $ \xi^+$, which is nonempty by Remark~\ref{Rem: (1, 20delta)_p1}. By the Morse Property (Theorem~\ref{Slim Q-Bigons infinite_p1}), $L$ is a quasiline, i.e. there exists a $(1,20\delta)$ quasi-isomtery $q: \Z\to L$ that is $M_{20\delta}$ coarsely surjective. Furthermore, $L$ is invariant under the subgroup of $\Isom X$ which stabilizes the set $\{\xi^-, \xi^+\}$. Therefore, to understand elementary subgroups it is key to understand actions on quasilines. In the spirit of the above, we shall consider quasilines $L$ that admit a $(1, \lambda)$ quasi-isometry $q: \Z\to L$ that is $M_\lambda$-coarsely onto. In Remark \ref{rem: Morse constant}, we fixed $M_\lambda$ the Morse constant for $(1, \lambda)$ quasigeodesics in a $\delta$-hyperbolic space.

Recall  that $(1, \lambda)$ quasigeodesics between the same end-points synchronously fellow travel by Corollary~\ref{cor: (1,mu)-fellow travel}, this is the key reason we consider such quasigeodesics. The following lemma generalizes this to the group setting and implies  that a group acting elliptically on a quasi-line and fixing the Gromov boundary pointwise has uniformly bounded orbits and is hence a tremble (see Theorem~\ref{thm:actionsclassrk1}).

Note that if $L$ is a quasiline then $\Isom L$ has a subgroup of index at most two that fixes the boundary pointwise. We shall denote this subgroup by $\Isom_{\! 0} L$. 

\begin{lemma}\label{lemma: Unif Bound}
Fix $\delta, \lambda\geq 0$. Let $L$ be a $\delta$-hyperbolic quasiline and $q: \Z \to L$  a $(1,\lambda)$ quasigeodesic. Then $q$ is $M$-coarsely surjective, where $M= M_\lambda$. Let $M'=M_{\lambda +2M_\lambda}$ be the Morse constant for $(1, \lambda + 2M_\lambda)$ quasigeodesics. Fix  $g\in \Isom_{\!0} L$. Setting $C_0(g) = 2M' + 8M + 5\lambda + d(gq(0), q(0))$ and $C(g) = 2M + C_0(g)$, we get that

 \begin{enumerate}
\item  If $n\in \Z$ then $d(g q(n), q(n))\leq C_0(g)$. 
\item If $x\in L$ then $d(g x, x)\leq C(g)$.
\end{enumerate}
In particular if $\G\to \Isom_{\!0} L$ is elliptic then $d(g x, x)\leq C_{ell}(\G)$ for every $g\in \G$ and $x\in L$, where $C_{ell}(\G) = 2M' + 10M + 5\lambda + B_0$ and $B_0:=\Sup{g\in \G}d(g q(0), q(0))<\8$.
\end{lemma}

\begin{proof}

(1) Observe that for $g\in \Isom_{\!0}L$ we can directly apply Corollary~\ref{cor: (1,mu)-fellow travel} with $c=gq$ and hence we set $C_0(g)=2M' + 8M + 5\lambda + d(gq(0), q(0))\geq 0$ so that for every $n\in \Z$ 
$$d(g. q(n), q(n))\leq C_0(g).$$

\noindent
(2) Let $x\in L$. There is an $n$ such that $d(x, q(n)) \leq M$ and since $g$ is an isometry, $d(gx, gq(n)) \leq M$. We deduce:
\begin{eqnarray*}
d(gx,x) &\leq& d(x, q(n))+ d(q(n), gq(n)) + d(gq(n), gx)
\\
&\leq&2M +  C_0(g)= C(g).
\end{eqnarray*}

\noindent
Now consider the case when $\G\to \Isom_{\!0} L$ elliptic. Then the constant $B_0$ defined in the statement is finite. Applying  part (1) uniformly to $q$ and all $gq$, $g\in \G$ we can set $C'_{ell}(\G)= 2M' + 8M + 5\lambda + B_0$ so that $d(g q(n), q(n)) \leq C'_{ell}(\G)$.

Let $x\in X$. Then, $d(x, q(n))\leq M$ for some $n\in \Z$. Letting $g\in \G$ and using the triangle inequality as above, we get 
\begin{align*}
    d(g x, x)     & \leq 2M+ C'_{ell}(\G) = C_{ell}(\G).
\end{align*}
\end{proof}

\begin{remark}\label{Rem: ell, fixing point on bound} 
    We note that the above proof also shows that if $\G\to \Isom X$ is elliptic and also fixes a point $\xi \in \partial X$, then there is an $\e>0$ so that the set $\O^\e(\rho(\G))$ is unbounded. Indeed, take a $(1, 10\delta)$ quasi-ray $q$ converging to $\xi$. Since $\G$ fixes $\xi$, the Hausdorff distance between $q$ and $gq$ for any $g \in \G$ is uniformly bounded. The arguments from the lemma can then be adapted to the quasi-ray $q$ to get the desired conclusion.  
\end{remark}

\subsection{Higher rank} 

Similar to the rank-one case, given   distinct boundary points $\xi_i^-, \xi_i^+\in \partial X_i$ for $i=1, \dots, \D$, we may consider  $L_i$ to be the union of all $(1, 20\delta)$ quasi-geodesics that limit to $\xi_i^-$ and $ \xi_i^+$. Set $\xi^-= (\xi_1^-, \dots, \xi_\D^-)$, and $\xi^+= (\xi_1^+, \dots, \xi_\D^+)$ and  form the product $\Prod{i=1}{\D}L_i\subset \X$, which is invariant under  the subgroup of $\Aut \X$ that stabilizes the set $\{\xi^-, \xi^+\}$. We can therefore take the product of these to form a $(1,20\delta)$ quasi-isomtery $q = \Prod{i=1}{\D} q_i: \Z^\D\to \Prod{i=1}{\D}L_i$ that is $M_{20\delta}$ coarsely surjective. Recall we are using the $\ell^\8$-product metric. Furthermore, $\Prod{i=1}{\D}L_i$ is invariant under the subgroup of $\Aut \X$ that stabilizes $\{\xi^-, \xi^+\}$ as well as the finite index subgroup  which fixes the factors $\{\xi_i^-, \xi_i^+\}$ for each $i=1, \dots, \D$.

 The following is immediate from Lemma~\ref{lemma: Unif Bound} applied to the factors. It implies  that an elliptic action on a product of quasilines preserving factors and fixing the boundary pointwise has uniformly bounded orbits and is hence a tremble. 
 \begin{cor}\label{cor: Unif Bound quasiflat}  Fix $\delta,\lambda\geq 0$. Let $L_1, \dots, L_\D$ be $\delta$-hyperbolic quasilines. Let $q=\Prod{i=1}{\D}q_i: \Z^{\D} \to \Prod{i=1}{\D}L_i$ be the product of  $(1,\lambda)$ quasigeodesic. Then $q$ is a $(1,\lambda)$ quasi-isometry that is  $M$-coarsely surjective, where $M=M_\lambda$ is the Morse constant for $(1, \lambda)$ quasigeodesics. Let $M'=M_{\lambda + 2M}$ be the Morse constant for $(1, \lambda + 2M)$ quasigeodesics. Fix  $g\in \Prod{i=1}{\D}\Isom_{\! 0}L_i$. Setting $C_0(g):= 2M' + 8M+ 5\lambda+ d(g q(0), q(0))$ and $C(g) := 2M + C_0(g)$, we have that
 
 \begin{enumerate}
\item If $a\in \Z^{\D}$ then $d(gq(a), q(a))\leq C_0(g)$.
\item If $x\in \Prod{i=1}{\D}L_i$ then $d(g x, x)\leq C(g)$.
\end{enumerate}
 In particular, if $\G\to \Prod{i=1}{\D}\Isom_{\! 0} L_i$ is elliptic, then $d(g x, x)\leq C_{ell}(\G)$ for every $g\in \G$ and $x\in \Prod{i=1}{\D}L_i$, where $C_{ell}(\G) = 2M' + 10M + 5\lambda +B_0$, and  $B_0:=\Sup{g\in \G}d(g q(0), q(0))<\8$.
\end{cor}

We are now ready  to prove the first step towards the amenability of $E_\G(\xi^-, \xi^+)$ when $\G$ is acting AU-acylindrically on $\X$. We note that to keep with our convention that subscripts denote components in products, we will use superscripts for sequences as we have done in previous sections. We hope that the reader will not confuse these for powers, as the only powers we take are inverses, specifically $g_a^{-1}$.

\begin{theorem}\label{theorem: elementary acyl is amenable}
Let $\G \to \Aut\X$ be  AU-acylindrical. If $\xi^-, \xi^+\in \partial_{reg}\X$ have distinct factors then  $E_\G(\xi^-, \xi^+)$ is amenable.
\end{theorem}

\begin{proof}
Recall that we are using the $\ell^\8$-product metric, and that $\set_\e\{\cdot\}$ denotes the coarse pointwise stabilizer. Since the action of $\G$ is AU-acylindrical on $\X$, so is the action of the subgroup $E_\G(\xi^-, \xi^+)$. To conserve notation, without loss of generality, we may assume $\G= E_\G(\xi^-, \xi^+)$. Further, let $M = M_{20\delta}$ be the associated Morse constant from the Morse Property (See Theorem~\ref{Slim Q-Bigons infinite_p1}), and let $\lambda= \max{}\{20\delta, M\}$. 

By Remark~\ref{Rem: (1, 20delta)_p1}   for each $i\in \{1,\dots, \D\}$ there exists a $(1,20\delta)$ quasigeodesic $q_i: \Z\to X_i$ with end points $\xi_i^-$ and $\xi_i^+$. Let $L_i$ be the union of all such quasigeodsics. We have that $L_i$ is a quasiline and $q_i$ is $M$ coarsely surjective. Let $\Fl=\Prod{i=1}{\D}L_i\subset \X$ be the corresponding product. Then $q=\Prod{i=1}{\D}q_i: \Z^\D\to \Fl$ is a $(1, 20\delta)$ quasi-isometry that is $M$ coarsely surjective. Note that $\G$ preserves $\Fl$. Let $M' = M_{\lambda + 2M}$ is the Morse constant for $(1, \lambda + 2M)$ quasigeodesics.

We claim that $\set_{2\lambda}\{q(0)\}\subset\set_{2M' + 8M + 7\lambda}\{q(0), q(a))\}$ for all $a\in \Z^\D$. Indeed, if $g\in \set_{2\lambda}\{q(0)\}$ then $d(gq(0), q(0))\leq 2\lambda$. So, by   Corollary~\ref{cor: Unif Bound quasiflat} (1), if  $a\in \Z^\D$ then  $d(gq(a), q(a))\leq C_0(g) \leq 2M' + 8M + 7\lambda$.

Now, since the action of $\G$ on $\X$ is AU-acylindrical and   $\Fl$ is $\G$ invariant, it follows that the restricted action of $\G$  to $\Fl$ is also AU-acylindrical. We now use this assumption: Let  $\e = 2M' + 8M + 7\lambda$. By AU-acylindricity, there exists an $R=R(\e)>0$ so that  for all $z \in \Z^{\D}$ if $d(q(0), q(z)) >R$ then the cardinality $|\set_\e\{q(0), q(z)\}|<\8$. Since $\Fl$ is unbounded, there exists a $z\in \Z^\D$ such  that $d(q(0), q(z)) >R$, which we now fix. Let $N_{\e}=N_{\e}(q(0),q(z))= |\set_\e\{q(0), q(z)\}|<\8$ and note that by the previous paragraph applied to our choice of $z$, it follows that

$$|\{g \in \G : d(gq(0), q(0)) \leq 2\lambda\}|\leq N_{\e}<\8.$$

Let $F_n = \Cup{a\in [-n,n]^{\D}}{}\{ g \in \G : d(gq(0), q(a))\leq \lambda \}$.  Since $\lambda\geq M$ and $q$ is $M$-coarsely surjective, it follows that $\Cup{n\in \N}{} F_n = \G$. To give an upper bound on the cardinality $|F_n|$ observe that  for each $a\in [-n,n]^{\D}$ either $\{ g \in\G: d(gq(0), q(a))\leq \lambda \}= \varnothing$ or there is a $g_{a}\in \{ g \in \G: d(gq(0), q(a))\leq \lambda \}$. If there is such a $g_{a}$, then for any $h\in \{ g \in \G: d(gq(0), q(a))\leq \lambda\}$, we get that 
\begin{eqnarray*}
   d(g_{a}^{-1}.h .q(0), q(0))&\leq& d(g_{a}^{-1}.h .q(0), g_a^{-1}q(a))+ d(g_a^{-1}.q(a), q(0)) \leq  2 \lambda .
\end{eqnarray*}

Therefore, $| g_a^{-1}\cdot \{ g \in \G: d(gq(0), q(a))\leq \lambda\}|\leq N_{\e}$ and we have shown by isometry of the action that
\begin{eqnarray*}
|F_n |\leq N_{\e}\cdot (2n+1)^{\D}.
\end{eqnarray*}

Note that the finite constant $N_{\e}$ above is independent of $n$. Since amenable groups are closed under taking unions, we now prove that every finitely generated subgroup of $\G$ has polynomial growth of degree bounded by $\D$ and is hence amenable. To this end, let $S\subset\G$ be a finite set and consider the associated finitely generated subgroup $H$. Then there is a $k^0\in \N$ so that  $S\subset F_{k^0}$, i.e. if $s^1, s^2 \in S$ then there are $a^1, a^2 \in [-k^0,k^0]^{\D}$ so that $d(s^i q(0), q(a^i))\leq \lambda$ for $i=1, 2$. 

Let $C_0(s)$ are the constants provided by Corollary~\ref{cor: Unif Bound quasiflat}. Since the action is by isometries, this implies that 
\begin{eqnarray*}
d(s^1s^2.q(0), q(a^1+ a^2))&\leq& d(s^1s^2.q(0), s^1.q(a^2)) + 
d(s^1.q(a^2), s^1.q(a^1+a^2)) \\
&&+ d(s^1.q(a^1+a^2), q(a^1+a^2))\\
&\leq& \lambda +\|a^1\|_\8+ \lambda + \max{s\in S}\, d(s.q(a^1 + a^2), q(a^1 + a^2))\\
&\leq&2\lambda +k^0 + \max{s\in S}\, C_0(s).
\end{eqnarray*}

Now by the coarse surjectivity, there exists  $b \in \Z^{\D}$ such that $d( s^1s^2.q(0), q(b)) \leq \lambda$. The following calculation shows that $s^1s^2 \in F_{4\lambda + 3k^0 + \max{s\in S}\, C_0(s)}$:
\begin{eqnarray*} 
\|b\|_\8 & = &\|a^1 + a^2 +b -a^1 -a^2\|_\8 \\
 & \leq & \|a^1 + a^2 -b\|_\8 + \|a^1 +a^2\|_\8 \\
& \leq & d( q(a^1+ a^2) , q(b)) + \lambda + 2k^0 \\
& \leq & d(q(a^1+ a^2), s^1s^2.q(0)) + d(s^1s^2.q(0), q(b)) + \lambda + 2k^0 \\
& \leq & 4\lambda + 3k^0 + \max{s\in S}\, C_0(s)
\end{eqnarray*} 

Fix $n\in \N$. We claim that if $s^1, \dots, s^n \in S$  then $s^1s^2\cdots s^n \in F_{ (3n -2)\lambda + 3k^0 + (n-1) \max{s\in S}\, C_0(s)}$, which we shall prove by induction on $n$. Obviously, we have established above that this holds for $n = 2$ (and holds for $n=1$ as well). Assume we have proven the claim for $s^1s^2\cdots s^n$, i.e. there exists $b \in \Z^{\D}$ such that $$d(s^1s^2\cdots s^n q(0), q(b)) \leq \lambda$$ and $\|b\|_\infty \leq (3n -2)\lambda + 3k^0 + (n-1)\max{s\in S}\, C_s$. 

Let $s^{n+1}\in S$ and consider $s^1s^2\cdots s^n s^{n+1}. q(0)$. Choose $b$ as above and $w \in \Z^{\D}$ such that $$d(s^1s^2\cdots s^n s^{n+1}. q(0), q(w)) \leq \lambda.$$ 

Then similar to the computation above, we deduce that \begin{eqnarray*} 
\|w\|_\8 &= &\|w +b-b\|_\8\\
& \leq &\|w-b\|_\8 + \|b\|_\8 \\
& \leq& \|w-b\|_\8 + (3n -2)\lambda + 3k^0 + (n-1)\max{s\in S}\, C_0(s)\\
& \leq &d(q(w), q(b)) + \lambda + (3n -2)\lambda + 3k^0 + (n-1)\max{s\in S}\, C_0(s)\\
& =& d(q(w), q(b)) + (3n -1)\lambda + 3k^0 + (n-1)\max{s\in S}\, C_0(s) \\
& \leq& d(q(w), s^1\cdots s^n s^{n+1}. q(0)) + d( s^1\cdots s^n s^{n+1}. q(0), s^1s^2\cdots s^n. q(0)) \\
 && + d( s^1s^2\cdots s^n. q(0), q(b)) + (3n -1)\lambda + 3k^0 + (n-1)\max{s\in S}\, C_0(s)  \\
& \leq &\lambda + d(s^{n+1}. q(0), q(0)) + \lambda + (3n -1)\lambda + 3k^0 + (n-1)\max{s\in S}\, C_0(s) \\
& \leq &(3n +1)\lambda + 3k^0 +  d(s_{n+1}. q(0), q(0))  + (n-1)\max{s\in S}\, C_0(s) \\
& \leq &(3n +1)\lambda + 3k^0 + n \hspace{2pt}\max{s\in S}\, C_0(s).
\end{eqnarray*}

This completes the induction argument, and gives us the linear relation that the set of words from $S$ of length bounded by $n$ is a subset of $F_{n(\lambda + \max{s \in S}C_0(s)) + \lambda'}$.  Therefore, $H$ has polynomial growth and is hence amenable. Finally $\G$ is the union of it's finitely generated subgroups, all of which have polynomial growth of exponent $\D$, and hence $\G$ is also amenable.
\end{proof}

\noindent
\textbf{\underline{Note}:} For the remainder of this subsection,  suppose that $\G$ is an amenable group with an  action $\G\to \Aut_{\!0}\Fl$, where $\Fl = \Prod{i=1}{\D}L_i$. Up to replacing $\G$ with a subgroup of index at most $2^\D$, we may assume that $\G$ acts trivially on $\partial_{reg} \Fl$. Fix $\xi\in \partial_{reg} \Fl$. We then obtain $\b_i: \G \to \R$, the Busemann homomorphism associated to $\xi_i\in \partial L_i$ for each $i =1, \dots, \D$  and hence $\b:=\Delta(\b_1, \dots, \b_\D): \G\to \R^\D$. This yields an action by translations $g(r)= r+\b(g)$, for $g\in \G$ and $r\in \R^{\D}$. In coordinates we have $$g.(r_1, \dots, r_{\D}) = (r_1 + \b_1(g), \dots, r_{\D} + \b_{\D}(g)).$$ 

We will show that AU-acylindricity of the original action on $\F$ descends to  AU-acylindricity for this action by translations $\R^{\D}$.

\begin{prop}\label{prop: acyldescends} Let $\G$ be amenable and $\G \to \Aut_{\!0}\Fl$  be an AU-acylindrical action with trivial action on $\partial_{reg}\Fl$. Then the action by translation on $\R^{\D}$ defined above is also AU-acylindrical.  \end{prop} 

\begin{proof} Fix  $q_i: \Z\to L_i$ a $(1, 20\delta)$ quasi-isometry for $i=1, \dots, \D$. Then $q= \Prod{i=1}{\D}: \Z^{\D}\to \Fl$ is a$(1, 20\delta)$ quasi-isometry as well.  Let $\e > 0$. Let $x= q(0)$ be a base point in $\Fl$. Since $\G$ preserves the factors,  for  each $i\in \{1, \dots, \D\}$  we may project $\G \to \Isom L_i$ and apply Proposition~\ref{prop:smalltranslation} and Lemma~\ref{Lem:shiftfn}, yielding the functions $A_i$  and associated constants $B_i, E_i$. Let $B = \max{i} \hspace{2pt}B_i$ and $E = \max{i} \hspace{2pt} E_i$ (with respect to the base point $x_i = q_i(0)$).

Let $R$ be the AU-acylindricity constant associated to the action of $\G \to \Aut_{\! 0} \Fl$ for $\displaystyle \e' = \e + 20\delta + B + E$. We claim that AU-acylindricity holds for $\e$ for the action on $\R^\D$ by taking $R_\e = 2\e$.

Indeed, let $r, s\in\R^{\D}$ at distance greater than $2\e$. Consider $g \in \G$ such that  $$d( r, gr) \leq \e \text{  and  } d(s, gs) \leq \e.$$

Since $\G$ is acting by translation, the above inequalities imply that $|\b_i(g_i)| \leq \e$ for all $i \in\{1, \dots, \D\}$. Applying Lemma~\ref{Lem:shiftfn} to each factor, it follows that in $\F$, $$d(x, gx) \leq \e + E + 20\delta.$$ 

Fix an element $h \in \G$ such that $d(h  x,  x) \geq R$. Such an element must exist as the orbit of $\G$ on $\F$ is unbounded. There are two possibilities to consider in each factor $L_i$. If $\b_i(h_i) \leq A_i(\e + 20\delta + E_i)$, then it follows from Proposition~\ref{prop:smalltranslation} that $$d(h_ix_i, g_ih_ix_i) \leq |\b_i(g_i)| + B_i \leq \e + B_i.$$
If not, then $\b_i(h_i) \leq A_i(n)$ for a sufficiently large $n \geq \e + 20\delta + E_i$. However, then $d(x_i, g_ix_i) \leq \e + 20\delta + E_i \leq n$, and Proposition~\ref{prop:smalltranslation} still implies $d( h_ix, g_ih_ix) \leq |\b_i(g_i)| + B_i \leq \e + B_i$.

Thus we have $$d(x, gx)  \leq \e + 20\delta + E \leq \e'$$ and $$d(gh x , hx)  \leq \e  + B \leq \e'.$$

Since the action is AU-acylindrical, the element $g$ has at most finitely many choices (which is a uniform $N' = N(\e')$ in the acylindrical case) and thus we are done. \end{proof}

\begin{proof}[Proof of Proposition~\ref{prop:elemsubvirtab}] By Theorem~\ref{theorem: elementary acyl is amenable}, we know that $E_\G(\xi^-, \xi^+)$ is amenable.  By Proposition~\ref{prop: acyldescends}, we see that the action via the Busemann quasimorphisms descends to an AU-acylindrical action on $\R^{\D}$ by translations. The result now follows by Lemma~\ref{lem: AU eucl is Zk}.

\end{proof}

We now turn to the proof of Proposition~\ref{prop: reg fixed implies reg pair}. Thinking of acylindricity as being a generalization of a cocompact lattice, this result is in line with saying that there are no unipotents in a cocompact lattice.

\begin{proof}[Proof of Proposition~\ref{prop: reg fixed implies reg pair}]  We recall our hypotheses:  $\rho:\G \to \Aut\X$ is acylindrical and not elliptic and $\xi \in \partial_{reg}  \X$. Without loss of generality, we may first pass to a finite index subgroup if necessary and assume that $\G$ is factor-preserving. For ease of notation, let $H = \stab_\G\{\xi\}$.

\noindent 
\textbf{Case 1:} For each $i=1, \dots, \D$, the action $\rho_i(H)$ is either elliptic or parabolic.

As acylindricity passes to subgroups, by Theorem~\ref{thm:elimmixedfactors}, at least one factor is elliptic. By Proposition~\ref{prop:removepara}, if $J$ is the set of indices for which $\rho_i(H)$ is elliptic, then $\Delta(\rho_i: i\in J): H\to \Prod{i\in J}{}\Isom X_i$ is acylindrical. 
By Remark~\ref{Rem: ell, fixing point on bound} and the acylindricity applied to a half-ray tending to each $\xi_i \in \partial X_i$ proves that $H$ is finite.

\noindent
\textbf{Case 2:} For each $i=1, \dots, \D$, the action $\rho_i(H)$ is neither elliptic nor parabolic.

Lemma~\ref{Lem:noqpinacyl} rules out the possibility that some factor action is quasiparabolic. Therefore, by Theorem~\ref{thm:actionsclassrk1}  for each $i=1, \dots, \D$ the action of $\rho_i(H)$ on $X_i$ is oriented lineal so $\rho_i(H)$ fixes exactly two points in $\partial X_i$. Let $\xi_i'$ be the other fixed point.

Let $\xi'=(\xi_1', \dots, \xi_\D')\in \partial_{reg}\X$. It follows that $\xi$ and $\xi'$ have distinct factors and $H =\fix_\G\{\xi, \xi'\}$.  Moreover, since acylindricity passes to subgroups, the action $\rho(H)$ on $\X$ is acylindrical and not elliptic. By Corollary~\ref{prop:elemsubvirtab}, $H$ is virtually $\Z^k$ where $1\leq k \leq \D$. 

\noindent
\textbf{Case 3:} The set of indices $I$ for which  $\rho_i(H)$ is neither elliptic nor parabolic is not empty.

 By Corollary~\ref{cor: clean up acyl}, the action $\Delta(\rho_i:i\in I):H\to \Prod{i\in I}{}\Isom X_i$ is  acylindrical. By Case 2, for each $i\in I$ there exists $\xi_i'\in \partial X_i\setminus \{\xi_i\}$ that is fixed by $\rho_i(H)$. Finally, for $i\notin I$, let $\xi_i'= \xi_i$. Observing that $|I|\leq \D$ and setting $\xi'=(\xi_1', \dots, \xi_\D')$ completes the proof as $H = \fix_\G\{\xi, \xi'\}$ in this case as well.

Lastly, the furthermore part if the statement of the proposition follows directly from the details of Case 2 above. \end{proof}

\begin{cor}
    Suppose that $\G \to \Aut\X$ is $AU$-acylindrical with all factors lineal. Then $\G$ is virtually  $\Z^k$ for $0 \leq k \leq \D$. 
\end{cor}
\begin{proof} Up to passing to a finite index subgroup, we may assume that $\G \to \Aut_{\~0} \X$ with all factors admitting oriented lineal factors. Then result now follows from the fact that each lineal factor has exactly two distinct limit points on $\partial X_i$ that are fixed by the whole group and applying Proposition~\ref{prop:elemsubvirtab}.  \end{proof}

The following is the second part of Theorem~\ref{intro:regelts}. Let $g\in \G$ and $H\leq \G$. We shall denote the centralizer of $g$ in $\G$ by $\mathfrak C_\G(g)$ and the normalizer of $H$ in $\G$ by $\mathfrak N_\G(H)$. 

\begin{cor}\label{cor: cent reg elmt} Let $\G\to \Aut\X$ be AU-acylindrical, and $\g\in \G$ a regular element. Let $\g^-, \g^+\in \partial_{reg}\X$ be the repelling, respectively attracting fixed points of $\g$ with coordinates $\g_i^-, \g_i^+\in \partial X_i$, for $i=1, \dots, \D$. Then the following are finite index inclusions $$\mathfrak C_\G(\g)\leq E_\G(\g^-, \g^+)\leq \mathfrak N_\G(E_\G(\g^-, \g^+))\leq \Sym_\X(\D)\ltimes \Cap{i=1}{\D}\stab_\G\{\g_i^-, \g_i^+\}.$$ 
\end{cor}

\begin{remark}\label{rem: conj lox fixed pts}
   Let $\g, g\in \Isom X$. If $\g$ is loxodromic  then it is straightforward to verify that $g\g g^{-1}$ is also loxodromic with repelling fixed point $g.\g^-$ and attracting fixed point $g.\g^+$. 
\end{remark}

\begin{proof} Recall that $E_\G(\g^-, \g^+)= \Cap{i=1}{\D}\fix_\G\{\g_i^-, \g_i^+\}$. As any subgroup normalizes itself, the second inclusion  $E_\G(\g^-, \g^+)\leq \mathfrak N_\G(E_\G(\g^-, \g^+))$ is immediate. 

Without loss of generality, assume that $\G$ is factor-preserving, as this  amounts to passing to a subgroup of index at most $\D!$. Let us prove that $\mathfrak N_\G(E_\G(\g^-, \g^+))\leq \Cap{i=1}{\D}\stab_\G\{\g_i^-, \g_i^+\}$ and so by the orbit stabilizer theorem the index of the second inclusion is at most $2^\D$.

Let $g\in \mathfrak N_\G(E_\G(\g^-, \g^+))$ and note that $g\g g^{-1}\in E_\G(\g^-, \g^+)$, meaning that $g_i\g_i g_i^{-1}$ fixes $\g_i^-$ and $\g_i^+$. By Remark~\ref{rem: conj lox fixed pts}, $g_i\g_i g_i^{-1}$ is a loxodromic with repelling and attracting fixed points $g.\g_i^-$ and $g.\g_i^+$, respectively, and therefore, $\{g_i.\g_i^-, g_i.\g_i^+\}=\{\g_i^-, \g_i^+\}$, for each $i=1, \dots, \D$ i.e. $g\in \Cap{i=1}{\D}\stab_\G\{\g_i^-, \g_i^+\}$.

The above argument also shows that if $g\in \mathfrak C_\G(\g)$ then $g\in E_\G(\g^-, \g^+)$, as in this case $g\g g^{-1}=\g$ and therefore, only the question of finite index remains. By Proposition~\ref{prop:elemsubvirtab}, $E_\G(\g^-, \g^+)$ is virtually free abelian. This completes the proof.
\end{proof}

\bibliographystyle{alpha}
\bibliography{AHR}

\end{document}